\documentclass[10pt]{amsart}

\usepackage{amssymb}
\usepackage{amsmath}
\usepackage{amsthm}
\usepackage{latexsym}
\usepackage{amssymb}
\usepackage{mathtools}
\usepackage[english]{babel}
\usepackage{enumerate}
\usepackage{graphicx} 
\usepackage{float}
\usepackage{mathrsfs}
\usepackage[left=2.5cm,right=2.5cm,top=2.75cm,bottom=2.75cm]{geometry}
\usepackage{tikz}
\usetikzlibrary{arrows,automata,shapes,calc,patterns}

\usepackage[colorlinks=true,urlcolor=green!50!black,
citecolor=red!60!black,linkcolor=blue!75!black,linktocpage,pdfpagelabels,
bookmarksnumbered,bookmarksopen]{hyperref}
\usepackage{ulem}

\newtheorem*{unthm*}{Theorem}
\newtheorem{thm}{Theorem}[section]
\newtheorem{prop}[thm]{Proposition}
\newtheorem{lem}[thm]{Lemma}

\theoremstyle{definition}

\theoremstyle{remark}
\newtheorem{remark}{Remark}[section]
\theoremstyle{notation}
\newtheorem*{notation}{Notation}

\numberwithin{equation}{section}

\newcommand{\R}{\mathbb{R}}
\newcommand{\RN}{\mathbb{R}^N}
\newcommand{\N}{\mathbb{N}}

\newcommand{\E}{\mathbb{E}}
\newcommand{\Cc}{\mathbb{C}}

\renewcommand{\epsilon}{\varepsilon}
\newcommand{\wh}{\widehat}

\newcommand{\cF}{{\mathcal F}}

\begin{document}

\title[Some remarks on a minimization problem associated to a $4$-NLS]{Some remarks on a minimization problem associated to a fourth order nonlinear Schr\"odinger equation}

\thanks{This work has been carried out in the framework of the project NONLOCAL (ANR-14-CE25-0013), funded by the French National Research Agency (ANR).}

\author[N. Boussa\"id, A. J. Fern\'andez and L. Jeanjean]{Nabile Boussa\"id, Antonio J. Fern\'andez and Louis Jeanjean}

\address{
\vspace{-0.4cm}
\newline
\textbf{{\small Nabile Boussa\"id}}
\newline \indent Laboratoire de Math\' ematiques (UMR 6623), Universit\' e Bourgogne Franche-Comt\'e,
\newline \indent 16, Route de Gray 25030 Besan\c con Cedex, France}
\email{nabile.boussaid@univ-fcomte.fr}

\address{
\vspace{-0.45cm}
\newline
\textbf{{\small Antonio J. Fern\'andez}} 
\newline \indent Universit\'e Polytechnique Hauts-de-France, EA 4015-LAMAV-FR CNRS 2956, F-59313 Valenciennes, France \&
\newline \indent Laboratoire de Math\'ematiques (UMR 6623), Universit\'e de Bourgogne Franche-Comt\'e, 
\newline \indent 16, Route de Gray 25030 Besan\c con Cedex, France \&
\newline \indent Department of Mathematical Sciences, University of Bath, Bath BA2 7AY, United Kingdom,}
\email{ajf77@bath.ac.uk}


\address{
\vspace{-0.45cm}
\newline
\textbf{{\small Louis Jeanjean}}
\newline \indent Laboratoire de Math\' ematiques (UMR 6623), Universit\' e Bourgogne Franche-Comt\'e,
\newline \indent 16, Route de Gray 25030 Besan\c con Cedex, France}
\email{louis.jeanjean@univ-fcomte.fr}

\begin{abstract} 
Let $\gamma > 0\,$, $\beta > 0\,$, $\alpha > 0$ and $0 < \sigma N < 4$. In the present work, we study, for $c > 0$ given, the constrained minimization problem
 \begin{equation*}
\label{MinL2fixed}
m(c):=\inf_{u\in S (c) }E(u),
\end{equation*}
where \begin{equation*}
E (u):=\frac{\gamma}{2}\int_{\R^N}|\Delta u|^2\, dx -\frac{\beta}{2}\int_{\R^N}|\nabla u|^2\, dx-\frac{\alpha}{2\sigma+2}\int_{\R^N}|u|^{2\sigma+2}\, dx, 
\end{equation*}
and
\begin{equation*} 
S(c):=\left\{u\in H^2(\R^N):\int_{\R^N}|u|^{2}\, dx=c\right\}.
\end{equation*}
The aim of our study is twofold. On the one hand, this minimization problem is related to the existence and orbital stability of standing waves for the mixed dispersion nonlinear biharmonic Schr\"odinger equation 
\begin{equation*}
i \partial_t \psi -\gamma \Delta^2 \psi - \beta \Delta \psi + \alpha |\psi|^{2\sigma} \psi =0, \quad \psi (0, x)=\psi_0 (x),\quad (t, x) \in \R \times \R^N.
\end{equation*}
On the other hand, in most of the applications of the Concentration-Compactness principle of P.-L. Lions, the difficult part is to deal with the possible dichotomy of the minimizing sequences. The problem under consideration provides an example for which, to rule out the dichotomy is rather standard while, to rule out the vanishing, here for $c > 0$ small, is challenging.  We also provide, in the limit $c \to 0$, a precise description of the behavior of the minima. Finally, some extensions and open problems are proposed. 
\end{abstract}

\maketitle
 
\section{Introduction}

In this work we consider a constrained minimization problem which is motivated by the search of standing waves solutions for the  biharmonic 
NLS (Nonlinear Schr\"odinger Equation) with mixed dispersion
\begin{equation}
\label{4nlsdis}
i \partial_t \psi -\gamma \Delta^2 \psi - \beta \Delta \psi + \alpha |\psi|^{2\sigma} \psi =0, \quad \psi (0, x)=\psi_0 (x),\quad (t, x) \in \R \times \R^N.
\end{equation}
Here $\gamma >0$, $\beta \in \R$ and $\alpha >0$ are given parameters and $0 < \sigma N < 4^*$ where we define $4^{\ast}:= 4N/(N-4)^{+}$, namely $4^{\ast} = +\infty$ if $N \leq 4$ and $4^{\ast} = 4N/(N-4)$ if $N \geq 5$.

\medbreak
About twenty years ago, equation~\eqref{4nlsdis} was introduced for several distinct physical motivations, see in particular~\cite{Ka, KaSh} and~\cite{FiIlPa}. It has been since then the object of intensive studies, some dealing with dynamical issues such as local or global well-posedness, others dealing with the existence and properties of certain kind of solutions. We refer to the introductions of the works~\cite{BoCaMoNa, BoCaGoJe, BoLe, NaPa} for a presentation of the more recent results concerning~\eqref{4nlsdis}.

\medbreak
Of particular interest are the so called standing waves solutions, i.e. solutions of the form $\psi (t,x)=e^{i\lambda t} u(x)$ with $\lambda >0$. The function $u$ then satisfies the elliptic equation
\begin{equation}\tag{$4$-NLS}\label{4nls}
\gamma \Delta^2 u + \beta \Delta u+\lambda u= \alpha |u|^{2\sigma}u, \quad u \in H^2(\R^N).
\end{equation}
A possible choice is to consider that $\lambda >0$ in~\eqref{4nls} is given and to look for solutions as critical points of the functional
\begin{equation} \label{J-quadr}
\E(u):=\frac{\gamma}{2}\int_{\R^N} |\Delta u|^2\, dx - \frac{\beta}{2} \int_{\R^N}|\nabla u|^2\, dx+\frac{\lambda}{2} \int_{\R^N}|u|^2\, dx - \dfrac{\alpha}{2\sigma +2}\int_{\R^N} | u|^{2\sigma +2}\, dx.
\end{equation}
Then, for physical reasons, one usually focuses on the so-called least energy solutions, namely solutions which minimize $\E$ on the set
\begin{equation} \label{ground-state}
\mathcal{N}:= \big\{u \in H^2(\R^N) \backslash \{0 \} : \E'(u) =0 \big\}.
\end{equation}
This is the approach followed in~\cite{BoCaMoNa,BoNa}. \medskip

Alternatively, one can consider the existence of solutions to~\eqref{4nls} having a prescribed $L^2$-norm. This way the parameter $\lambda$ is viewed as a Lagrange multiplier. For $c>0$ given, we consider the problem of finding solutions to 
\begin{equation}\label{Pc}\tag{$P_c$}
\gamma \Delta^2 u + \beta \Delta u +  \lambda u = \alpha |u|^{2 \sigma}u,  \quad u \in H^2(\RN) \quad \mbox{ with } \quad  \int_{\R^N}|u|^2 dx = c.
\end{equation}
 
\noindent It is standard to show that a critical point of the energy functional
\begin{equation}\label{def:E}
E (u):=\frac{\gamma}{2}\int_{\R^N}|\Delta u|^2\, dx - \frac{\beta}{2}\int_{\R^N}|\nabla u|^2\, dx-\frac{1}{2\sigma+2}\int_{\R^N}|u|^{2\sigma+2}\, dx,
\end{equation}
restricted to
\begin{equation}\label{de:Mmu}
S(c):=\left\{u\in H^2(\R^N):\int_{\R^N}|u|^{2}\, dx=c\right\},
\end{equation}
corresponds to a solution of~\eqref{Pc}. The value of $\lambda \in \R$ in~\eqref{Pc} is then an unknown of the problem. 

\medbreak
In~\cite{BoCaMoNa, BoCaGoJe} the authors study this problem assuming that $\beta <0$. First, in~\cite{BoCaMoNa} the mass subcritical case $0 < \sigma N < 4$ was considered. In this case, the functional $E$ is bounded from below on $S(c)$ for any $c>0$. Hence, it is possible to search for a critical point of $E$ restricted to $S(c)$ as a global minimizer. Setting, for $c>0$ given,
\begin{equation}\label{MinL2fixed}
m(c):=\inf_{u\in S (c) }E(u),
\end{equation}
the following  result was obtained.
\begin{thm}[{\protect\cite[Theorem 1.1]{BoCaMoNa}}]
\label{Compact-Min-Sol} 
Assume $\gamma > 0$, $\beta \leq 0$ and $\alpha > 0$. If $0<\sigma N <2$, then $m(c)$ is achieved for every $c>0$. If $2 \leq \sigma N <4$  then there exists a critical mass $\tilde{c} \in [0, \infty)$ such that 
\begin{itemize}
\item[i)] $ m(c)$ is not achieved if $c< \tilde c$;
\item[ii)] $ m(c)$ is achieved if $c > \tilde c$ and $\sigma = 2/N$;
\item[iii)] $ m(c)$ is achieved if $c \ge \tilde c$ and $\sigma \neq 2/N$.
\end{itemize}
Moreover if $\sigma$ is an integer and $m(c)$ is achieved, then there exists at least one radially symmetric minimizer.
\end{thm}

\begin{remark} $ $ \label{minimal mass} 
Let us point out that, for $\beta = 0$, it follows that $\tilde{c} = 0$ while, for $ \beta < 0$, it holds that $\tilde{c} > 0$. The appearance of a critical mass when $\beta < 0$ and $2 \leq \sigma N < 4$ is linked to the fact that each term of $E$ behaves differently with respect to scaling. Such a phenomenon was first observed in~\cite{CoJeSq} and has now been revealed in several distinct settings, see for instance~\cite{CaDoSaSo,DoSeTi,JeLu} for related results.
\end{remark}

\medbreak
In~\cite{BoCaGoJe} the authors considered, still assuming $\beta <0$, the mass critical case $\sigma N=4$ and the mass supercritical case $ 4 < \sigma N < 4^*$.  In particular, it was shown in~\cite[Theorem 1.2]{BoCaGoJe} that standing waves do not exist if $\sigma N=4$ and that, assuming $c>0$ sufficiently small, they do exist when $ 4 < \sigma N < 4^*$, see~\cite[Theorem 1.3]{BoCaGoJe}. Note that in the mass supercritical case $4 < \sigma N < 4^{\ast}$ the functional $E$ is no longer bounded from below on $S(c)$. The critical points obtained in~\cite{BoCaGoJe} are of saddle point type. In ~\cite{BoCaGoJe}, some multiplicity results for radial solutions were also derived.
\medbreak
Very recently, the case where $\beta >0$ started to be considered in the work~\cite{LuZhZh} which is devoted to the mass subcritical case $0 < \sigma N < 4$ and the mass critical case $\sigma N =4$. We shall come back later, in some details, to this work. 

\medbreak
Here we also deal with the case $\beta >0$ and restrict ourselves to the mass subcritical case $0 < \sigma N <4$, see however Section~\ref{sect-extensions} for a result in the mass critical case $\sigma N =4$.  Our first main result reads as follows
\medbreak
\begin{thm}\label{th1}
Assume $\gamma >0$, $\beta >0$, $\alpha >0$ and  $0<\sigma N <4$. For $m(c)$ defined as in~\eqref{MinL2fixed}, there exists $c_0 \in [0, \infty)$ such that:
\begin{itemize}
\item[i)] If $c > c_0,\,$ any minimizing sequence of $m(c)$ is precompact in $H^2(\R^N)$ up to translations. In particular, $m(c)$ is achieved.
\item[ii)] If $0 < \sigma < \max \left\{\frac{4}{N+1},1\right\}$, then we have that $c_0 = 0$.
\end{itemize}
In addition, if $u \in S(c)$ is a minimizer of $m(c)$, the associated Lagrange multiplier $\lambda \in \R$  satisfies $\lambda > \frac{\beta^2}{4 \gamma}.$
\end{thm}

\begin{remark} $ $ \label{stability} \label{symmetry}
\begin{itemize}
\item[a)] In the case where $\beta < 0$ and $2 \leq \sigma N < 4$, see Theorem~\ref{Compact-Min-Sol} (\cite[Theorem 1.1]{BoCaMoNa}), there exists a critical mass $\tilde{c} > 0$ such that $m(c)$ is not achieved if $c < \tilde{c}$. The situation is now different. In the case $\beta > 0$ considered in Theorem~\ref{th1}, except for $N=1$, if $\sigma \in( \frac{2}{N}, \max \left\{\frac{4}{N+1},1\right\})$ such that, for every $c > 0$, $m(c)$ is achieved. 
\item[b)] It is known from~\cite{Pa}, see also~\cite{BeKoSa}, that the Cauchy problem associated to~\eqref{4nlsdis} is locally well-posed in $H^2(\R^N)$ as soon as $0 < \sigma N <4^*$. Also, in the mass subcritical case $0 < \sigma N <4$ that we are considering in Theorem~\ref{th1}, the global existence for the Cauchy problem holds, see~\cite{FiIlPa, Pa}. Thus, having at hand the precompactness up to translations of any minimizing sequence, it is standard to show the orbital stability of the set of global minima following the strategy laid down in~\cite{CaLi}.
\item[c)] When $\sigma >0$ is an integer, using a very recent result of L. Bugiera, E. Lenzmann and J. Sok~\cite{BuLeSo}, it is possible to obtain symmetry properties for the global minimizer of $m(c)$. In view of Theorem~\ref{th1}, we shall benefit from these results when $N=1,2$. More details will be given in Section~\ref{sect-extensions}.
\end{itemize}
\end{remark}
 
Let us now provide some elements of the proof of Theorem~\ref{th1}. First, assuming that $0 < \sigma N <4$, it is straightforward to show that the functional $E$ is bounded from below on $S(c)$ and coercive, see Lemma~\ref{coercive}. In particular, $m(c)$ is well defined for any $c>0$. Then, using a convenient version of the Concentration Compactness principle of P.-L. Lions, 
we deduce that, for any $c>0$, either the vanishing of a minimizing sequence occurs or it is precompact up to translations, see Lemma~\ref{alternative}. Namely, the ruling out of the vanishing also exclude the possibility of dichotomy for the minimizing sequences. Then, in Lemma~\ref{lem:existence}, we show that a necessary and sufficient condition to avoid the vanishing is that 
\begin{equation}\label{strict}
m(c) < - \frac{\beta^2}{8 \gamma}c.
\end{equation}
This condition is derived through the study of an associated minimization problem, see Section~\ref{sect-associated-problem}, which has also an interest by itself. More precisely, for all $c > 0$, we consider
\begin{equation} \label{mI-intro}
m_I(c) := \inf_{u \in S(c)}I(u),
\end{equation}
where
$$I(u) :=\frac{\gamma}{2}\int_{\R^N}|\Delta u|^2\, dx -\frac{\beta}{2}\int_{\R^N}|\nabla u|^2\, dx,$$
and we show that $m_I(c) = - \frac{\beta^2}{8 \gamma}c$, that the infimum is never achieved and that any minimizing sequence is vanishing. See Lemma~\ref{Mihai-1}.

\medbreak
When $c >0$ is large, it is quite straightforward to show that~\eqref{strict} holds. However, when $c>0$ is small, the situation is surprisingly much more involved and the treatment of this case is a central part of the present work. Under the assumption $0 < \sigma < \min\{\max\{4/(N+1),1\}, 4/N\}$, we manage to check~\eqref{strict} for all $c > 0$ through the construction of suitable  sequences of test functions. We refer to Section~\ref{sect-test-functions} for more details. Our choice of test functions is inspired by the following result which provides a description of the behavior of the minima, when they exist, as $c \to 0$.

\begin{thm}\label{th2}
Assume $\gamma >0$, $\beta >0$, $\alpha >0$ and $0<\sigma N <4$. Let $\{(u_n, c_n)\} \subset S(c_n) \times \R$ be such that $c_n \to 0$ as $n \to \infty$ and $u_n \in S(c_n)$ be a minimizer of $m(c_n)$ for each $n \in \N$. Then:
\begin{itemize} 
\item[i)] There exists $\{\varepsilon_n\} \subset \R^+$ with $\varepsilon_n \to 0$ such that
$$ -\frac{\beta^2}{8\gamma} (1+\epsilon_n) c_n = (1+\epsilon_n) m_I(c_n) \leq m(c_n) \leq m_I(c_n) = - \frac{\beta^2}{8 \gamma}c_n, \quad \forall\ n \in \N.$$
\item[ii)]  For $\lambda_n \in \R$  the Lagrange multiplier associated to $u_n$, 
it follows that
\begin{equation*}
\lambda_n \to \Big(\frac{\beta^2}{4 \gamma}\Big)^+ \quad \mbox{as} \quad n \to \infty.
\end{equation*}
\item[iii)] It holds that
$$ \frac{\|\Delta u_n\|_2}{\|u_n\|_2} \rightarrow \frac{\beta}{2 \gamma} \quad \mbox{and} \quad \frac{\|\nabla u_n\|_2}{\|u_n\|_2} \rightarrow \sqrt{\frac{\beta}{2 \gamma}} \quad \mbox{as} \quad n \to \infty.$$
\item[iv)]  Setting $v_n = \displaystyle \frac{u_n}{\|u_n\|_2}$ we have that $v_n \to 0$ in $L^p(\R^N)$ for any $p \in (2, 4^*)$. Also 
\begin{equation}
\label{4}
\int _{\R^N} \Big( |\xi |^2 - \frac{\beta}{2 \gamma} \Big)^2 | \wh{v_{n} }( \xi )|^2 \, d \xi \to 0 \quad \mbox{as} \quad n \to \infty.
\end{equation}
\end{itemize}
\end{thm}

\begin{remark} \label{remark test functions}$ $
\begin{itemize}
\item[a)] From i) we infer that the slope of $m(c)$ at the origin is precisely $-\frac{\beta^2}{8\gamma}$. Hence the condition~\eqref{strict} is delicate to check at $c > 0$ small. 
\item[b)] From~\eqref{4} we see that the $L^2-$norm of $\{\wh{v}_{n}\} \subset S(1)$ concentrates asymptotically on the sphere of radius $ \sqrt{\beta/2\gamma}$ 
centered at the origin as $n \to \infty$. This indicates one of the main features of the test functions to choose in order to prove that, under the assumption $0 < \sigma < \max\{4/(N+1),1\}$, the strict inequality~\eqref{strict} holds for every $c > 0$.
\end{itemize}
\end{remark}

Let us now turn back to the work~\cite{LuZhZh} and try to articulate our results with the ones presented in this work. 
In~\cite{LuZhZh}, in the mass subcritical case and mass critical cases, assuming that the non vanishing holds it is shown that any minimizing sequence is precompact, up to translation (and thus that the set of global minima is orbitally stable, see Remark~\ref{stability} b)). Instead of using the approach laid down by P.-L. Lions, the authors rely on a Profile Decomposition of bounded sequences in $H^2(\R^N)$ which was established in~\cite{ZhZhYa}. In~\cite{LuZhZh} is derived an explicit lower bound on $c>0$ above which the non-vanishing holds (this corresponds in~\cite{LuZhZh} to the case $\mu <0, |\mu|$ small). There are no results in~\cite{LuZhZh} when $c>0$ is small.  
\medbreak

We conclude our work, in Section~\ref{sect-extensions}, with some remarks and open problems. First we show, when $\sigma \in \N$, that properties of the minima of $m(c)$ can be obtained exploiting a result from~\cite{BuLeSo}. 
Also, it should be clear that the condition $0 < \sigma < \max\left\{4/(N+1),1 \right\}$ is the consequence of two particular trials of test functions. Nothing guarantees that we have obtained, in Theorem~\ref{th1}, the sharpest conditions for the existence of a minimizer for $c>0$ and deriving necessary and sufficient conditions to insure that it is the case is worth of study. 
Next, we indicate how our test functions also prove useful in the case where the problem is mass critical. Finally, we note that the existence of a minimizer of $m(c)$ for any $c>0$ small can be interpreted as a bifurcation result, in the $H^2(\R^N)$ norm, from the bottom of the essential spectrum of the operator 
$u \mapsto  \gamma  \Delta^2 u  + \beta \Delta u$ defined on $H^2(\R^N)$.  It would be interesting to see if such phenomena is also present for the ground states solutions obtained from the functional $\mathbb{E}$ in~\cite{BoCaMoNa, BoNa}. 

\medbreak
We now describe the organization of the work. In Section~\ref{preresults}, we present some preliminary results. In particular, we present the proof of Lemma~\ref{alternative} which shows that, if the vanishing do not occurs, then $m(c)$ is reached.  In Section~\ref{sect-associated-problem}, we study in details the associated minimization problem~\eqref{mI-intro}. In Section~\ref{sect-existence}, we derive the sufficient and necessary condition~\eqref{strict} which guarantees that the vanishing does not occur. Section~\ref{sect-test-functions} is devoted to the construction of our two families of test functions which permit to rule out the vanishing under the conditions on $\sigma >0$ given in Theorem~\ref{th1}. At this point the proof of Theorem~\ref{th1} is completed. In Section~\ref{sect-limit}, we give the proof of Theorem~\ref{th2}, which deals with the behaviour of the minima as $c \to 0$. Finally, in Section~\ref{sect-extensions}, we present some additional results and state some open problems. \medbreak

In the rest of the work, we assume that $N \geq 1,$ $\gamma > 0$, $\beta > 0$, $\alpha > 0$ and $0 < \sigma N < 4$.

\begin{notation}
For  $1\leq p<\infty,$ we denote by $L^p(\R^N)$ the usual Lebesgue space with norm
$$
\|u\|_p^p := \int_{\R^N}|u|^p\,dx.
$$
The Sobolev space $H^2(\R^N)$ is endowed with its standard norm
$$
\|u\|^2 := \int_{\R^N}|\Delta u|^2 + |\nabla u|^2+|u|^2 \, dx.
$$
 We denote by $'\rightarrow'$, respectively by $'\rightharpoonup'$, the strong convergence, respectively the  weak convergence in corresponding space and denote by $B_R(x)$ the ball in $\R^N$ of center $x$ and radius $R>0.$  We use the notation $o_n(1)$ for any quantity which tends to zero as $n \to \infty$.
we shall denote by $C>0$ a constant which may vary from line to line but which is not essential in the analysis. 

Eventually, we use the Fourier transform on $L^1(\R^n)$
\[ \hat{u}(\xi) = \mathcal{F}[u](\xi) = \int_{\R^N} e^{-i\xi x} u(x) \, d x.\]
\end{notation}

\section{Preliminary results} \label{preresults}

We shall make use of some inequalities that we now present. 
First, we recall   
the Gagliardo-Nirenberg inequality (see~\cite[Theorem in Lecture II]{Nir}): 
for all $0 \leq \sigma < 4/(N-4)^{+}$,  i.e. $0 \leq \sigma$ if $N \leq 4$ and $0 \leq \sigma < 4/(N-4)$ if $N \geq 5$, there exists a constant $B_{N}(\sigma) > 0$ such that
\begin{equation}\label{G-N-H2-ineq}
 \|u\|^{2\sigma+2}_{2\sigma+2}\leq B_N(\sigma)\,\|\Delta u\|^{\frac{\sigma N}{2}}_2\|u\|^{2+2\sigma-\frac{\sigma N}{2}}_2, \quad \forall\ u \in H^2(\R^N).
\end{equation}
Having at hand~\eqref{G-N-H2-ineq}, by interpolation and using the Sobolev inequality, one may infer (see~\cite[Theorem in Lecture II]{Nir}) that, for all $0 \leq \sigma < 2/(N-2)^{+}$, namely $0 \leq \sigma$ if $N \leq 2$ and $0 \leq \sigma < 2/(N-2)$ if $N \geq 3$, there exists a constant $C_N(\sigma)>0$ such that 
\begin{equation}\label{G-N-H1-ineq}
 \|u\|^{2\sigma+2}_{2\sigma+2}\leq C_N(\sigma)\|\nabla u\|^{{\sigma N}}_2\|u\|^{2+\sigma(2-N)}_2, \quad \forall\ u \in H^2(\R^N).
\end{equation}
We also use the following interpolation inequality,
\begin{equation}\label{interpolation}
\int_{\R^N} |\nabla u|^2 dx \leq \left( \int_{\R^N}|\Delta u|^2 dx \right)^{\frac{1}{2}} \left( \int_{\R^N}| u|^2 dx \right)^{\frac{1}{2}}, \quad \forall\ u \in H^2(\R^N).
\end{equation}
\medskip

\begin{lem}\label{lower-bound}
For any $\gamma >0$, $\beta \in \R$ and $u \in H^2(\R^N)$, we have
\begin{equation}\label{infimum}
\inf_{u \in H^2(\R^N)} \left( \frac{\gamma \|\Delta u\|_2^2 - \beta \|\nabla u\|_2^2}{\|u\|_2^2} \right) \geq  -\frac{\beta^2}{4 \gamma}.
\end{equation}
\end{lem}

\begin{proof}
It directly follows from~\eqref{interpolation} that
\begin{align*}
\gamma \|\Delta u\|_2^2 - \beta \|\nabla u\|_2^2 + \frac{\beta^2}{4 \gamma}\|u\|_2^2   
& \geq  \gamma \| \Delta u \|_2^2 - \beta \|\Delta u\|_2 \|u\|_2 + \frac{\beta^2}{4\gamma} \|u\|_2^2 =   \Big(  \sqrt{\gamma} \|\Delta u\|_2 - \frac{\beta}{2 \sqrt{\gamma}}\|u\|_2 \Big)^2 \geq 0
\end{align*}
getting, thus, the sought inequality.
\end{proof}

\begin{lem}\label{coercive}
Assume $\gamma >0$, $\beta \in \R$, $\alpha \in \R$ and $0<\sigma N <4$. 
The functional $E$ is coercive on $S(c)$ and in particular $m(c) > - \infty$ for any $c>0$.
\end{lem}

\begin{proof}
The claim follows directly using~\eqref{G-N-H2-ineq} and arguing for example as in Lemma~\ref{lower-bound}.
\end{proof}

Let us now introduce a scaling  useful to the rest of the work: for any $u \in S(c)$ and any $s>0$, 
\begin{equation}\label{dilatation-louis}
u_{s}(x):=s^{\frac{N}{4}}u(\sqrt{s} x).
\end{equation}
This definition is clearly motivated by the fact that $\|u_s\|_2=\|u\|_2$ for all $s > 0$. One then easily obtain that
\begin{equation}\label{functional-louis}
E(u_{s})=\frac{\gamma s^{2}}{2}\int_{\R^N}|\Delta u|^2\, dx -\frac{\beta s}{2}\int_{\R^N}|\nabla u|^2\, dx-\frac{\alpha s^{\sigma N/2}}{2\sigma+2}\int_{\R^N}|u|^{2\sigma+2}\, dx, \quad \forall\ s > 0.
\end{equation}
Having at hand this suitable rescaling, we prove several properties of $m(c)$  to rule out the dichotomy of the minimizing sequences. 
\begin{lem}\label{prop1} Assume $\gamma >0$, $\beta >0$, $\alpha >0$ and $0<\sigma N <4$. 
\begin{itemize}
\item[i)] $m(c) < 0,\ \forall\ c>0$;
\item[ii)] $m(\tau c) \leq \tau m(c),\ \forall\ \tau > 1,\ \forall\ c>0$;
\item[iii)] Assume that there exists a global minimizer $u \in S(c)$ of $m(c)$ for some $c>0$. Then $m(\tau c) < \tau m(c),\, \forall\ \tau >1$;
\item[iv)] $m(c_1 + c_2) \leq m(c_1) + m(c_2),\ \forall\ c_1, c_2 >0$;
\item[v)] Assume that there exists a global minimizer $u \in S(c_1)$ with respect to $m(c_1)$ for some $c_1 >0$ and 
let $c_2 >0$. Then $m(c_1 + c_2) < m(c_1) + m(c_2)$.
\end{itemize}
\end{lem}

\begin{proof}
i) Taking an arbitrary $u \in S(c)$ and considering $u_s$ as defined in~\eqref{dilatation-louis}, we see from~\eqref{functional-louis} that $E(u_s) \to 0^-$ as $s \to 0$ and i) follows.
\smallbreak
\noindent ii) For any $\varepsilon >0$, there exists $u \in S(c)$ such that $E(u) \leq m(c) + \varepsilon$. 
Defining $\widetilde{u}(x) = \tau^{\frac{1}{2}} u(x)$ we observe that
$$ \|\widetilde{u}\|_2^2 = \tau \|u\|_2^2 = \tau c; \quad \|\Delta \widetilde{u}\|_2^2 = \tau \|\Delta u\|_2^2; \quad 
 \|\nabla \widetilde{u}\|_2^2 = \tau \|\nabla u\|_2^2 \quad \textup{ and } \quad  \| \widetilde{u}\|_{2 \sigma +2}^{2 \sigma +2} = \tau^{\sigma +1} \|u\|_{2 \sigma +2}^{2 \sigma +2}.$$ 
Hence, we have that
\begin{align}\label{equation-x}
m(\tau c) \leq E(\widetilde{u})  
& = \tau \Big[ \frac{\gamma}{2} \|\Delta u\|_2^2 - \frac{\beta}{2} \|\nabla u\|_2^2 - \frac{\alpha \tau^{\sigma }}{2 \sigma +2}  \|u\|_{2 \sigma +2}^{2 \sigma +2}  \Big]  \nonumber \\
&\ \ < \tau \Big[ \frac{\gamma}{2} \|\Delta u\|_2^2 - \frac{\beta}{2} \|\nabla u\|_2^2 - \frac{\alpha}{2 \sigma +2}  \|u\|_{2 \sigma +2}^{2 \sigma +2}  \Big] \\
&\ \  = \tau E(u) \leq \tau (m(c) +  \varepsilon). \nonumber 
\end{align}
Since $\varepsilon>0$ is arbitrary, we see that ii) holds.
\smallbreak
\noindent iii) If $m(c)$ is achieved, for example, at some $u \in S(c)$, then we can set $\varepsilon=0$ in~\eqref{equation-x} and thus the strict inequality follows.
\smallbreak
\noindent iv) Assume, without loss of generality, $0 < c_2 \leq c_1$. Then, by ii), we have that
\begin{align*}
m(c_1 + c_2) & \leq  \frac{c_1 +c_2}{c_1} m(c_1)  = m(c_1) + \frac{c_2}{c_1}m(c_1) =  m(c_1) + \frac{c_2}{c_1}m(\frac{c_1}{c_2}c_2) \\
&\leq  m(c_1) + \frac{c_2}{c_1} \frac{c_1}{c_2} m(c_2) = m(c_1) + m(c_2).
\end{align*}
\smallbreak
\noindent v) The proof follows the lines of the one of iii) using iv) instead of i).
\end{proof}

\begin{lem}\label{lemma5}
Let $\{u_n\} \subset H^2(\R^N)$ be a bounded sequence such that $\|u_n\|_2^2 \to c >0$ and let $\widetilde{u}_n = \displaystyle c^{\frac{1}{2}} \frac{u_n}{\|u_n\|_2}.$  Then
\begin{itemize}
\item[i)] $\widetilde{u}_n \in S(c)$, $\forall n \in \N$.
\item[ii)] $\lim_{n \to \infty} |E(u_n) - E(\widetilde{u}_n)| = 0 $.
\end{itemize}
\end{lem}

\begin{proof}
Point i) is immediate. On point ii), we directly observe that
$$ \|\Delta \widetilde{u}_n\|_2^2 = \frac{c}{\|u_n\|_2^2} \,  \|\Delta u_n\|_2^2; \quad 
\|\nabla \widetilde{u}_n\|_2^2 = \frac{c}{\|u_n\|_2^2} \,  \|\nabla u_n\|_2^2 \quad \textup{ and } \quad  \|\widetilde{u}_n\|_{2 \sigma +2}^{2 \sigma +2} = \Big( \frac{c}{\|u_n\|_2^2} \Big)^{\sigma +1} \,  
\| u_n\|_{2 \sigma +2}^{2 \sigma +2}.$$
Since $\lim_{n \to \infty} \frac{c}{\|u_n\|_2^2}=1$ and $\{u_n\} \subset H^2(\R^N)$ is bounded, the result follows.
\end{proof}

The next statement ensures that either a minimizing sequence is vanishing or it is precompact up to translations. In other words, the non-vanishing rules out the possibility of dichotomy.

\begin{lem}\label{alternative}
Let $c>0$. If $\{u_n\} \subset S(c)$ is a minimizing sequence with respect to $m(c)$ then one of the following alternative holds:
\begin{itemize}
\item[i)] (vanishing) For all $R >0$, $$\lim_{n \to  \infty} \sup_{y \in \R^N} \int_{B_{R}(y)} |u_n|^2 dx  =0.$$
\item[ii)] (compactness up to translations)  There exists  (up to a subsequence) $u \in S(c)$ and a family $\{y_n\} \subset \R^N$ such that 
$u_n(\cdot - y_n) \to u$ in $H^2(\R^N).$ In particular $u$ is a global minimizer.
\end{itemize}
\end{lem}

\begin{proof}
Suppose that $\{u_n\} \subset S(c)$ is a minimizing sequence with respect to $m(c)$ which does not satisfy i). Then, there 
exists $R_0 >0$ such that  (up to a subsequence)
$$ 0 < \lim_{n \to  \infty} \sup_{y \in \R^N} \int_{B_{R_{0}}(y)} |u_n|^2 dx \leq c,$$
and  there exists (up to a subsequence) a family 
$\{y_n\} \subset \R^N$ such that
\begin{equation}\label{eq:0}
0 < \lim_{n \to  \infty} \int_{B_{R_{0}}(y_n)} |u_n(x-y_n)|^2 dx \leq c.
\end{equation}
Since $\{u_n\}$ is a minimizing sequence, by Lemma~\ref{coercive}, we deduce that $\{u_n\}$  is bounded in $H^2(\RN)$ and so (up to a subsequence) there exists $u \in H^2(\R^N)$ such that
\begin{equation}\label{convergence-1}
u_n(\cdot -y_n) \rightharpoonup u \mbox{ in } H^2(\R^N) \quad \mbox{and} \quad u_n(\cdot -y_n) \to u \mbox{ in } L_{loc}^p(\R^N), \quad \mbox{ for } 1 \leq p < \frac{2N}{(N-4)^+}.
\end{equation}
Using Rellich-Kondrashov theorem, observe that~\eqref{eq:0} implies that $u \not \equiv 0$. Now, we define $v_n := u_n(\cdot - y_n) - u$ and, by~\eqref{convergence-1}, we have that $v_n \rightharpoonup 0$ in $H^2(\R^N)$ and so, that

\begin{equation*}
\|\Delta u_n\|_2^2 = \|\Delta (u+v_n)\|_2^2 = \|\Delta u\|_2^2 + \|\Delta v_n\|_2^2 + o_n(1),
\end{equation*}
\begin{equation*}
\|\nabla u_n\|_2^2 = \|\nabla (u+v_n)\|_2^2 = \|\nabla u\|_2^2 + \|\nabla v_n\|_2^2 + o_n(1),
\end{equation*}
and 
\begin{equation}\label{eq13}
\|u_n\|_2^2 = \|u+v_n\|_2^2 = \| u\|_2^2 + \|v_n\|_2^2 + o_n(1).
\end{equation}
On the other hand, by the Brezis-Lieb lemma~\cite[Theorem 1]{BrLi},
\begin{equation*}
\|u_n\|_{2 \sigma +2}^{2 \sigma +2} = \|u+v_n\|_{2 \sigma +2}^{2 \sigma +2} = \| u\|_{2 \sigma +2}^{2 \sigma +2} +
 \|v_n\|_{2 \sigma +2}^{2 \sigma +2} + o_n(1).
\end{equation*}
Hence, we have that 
\begin{equation}\label{eq11}
E(u_n) = E(u_n(\cdot -y_n)) = E(u+v_n) = E(u) + E(v_n) + o_n(1).
\end{equation}
\bigbreak
\noindent \textbf{Claim: } $\|v_n\|_2^2 \to 0$ \textit{ as } $n \to \infty$.
\smallbreak In order to prove this, let us denote $c_1 = \|u\|_2^2 > 0$. By~\eqref{eq13}, if we show that $c_1 =c$, the claim follows. We assume by contradiction that $c_1 <c$ and we define
$$\widetilde{v}_n = \frac{(c-c_1)^{\frac{1}{2}}}{\|v_n\|_2}\,v_n .$$
By Lemma~\ref{lemma5} and~\eqref{eq11}, it follows that
$$E(u_n) = E(u) + E(v_n) + o_n(1) = E(u) + E(\widetilde{v}_n) + o_n(1) \geq E(u) + m(c - c_1) + o_n(1).$$
Hence,  by Lemma~\ref{prop1}, iv), we have that
\begin{equation}\label{eq15}
m(c) \geq E(u) + m(c- c_1) \geq m(c_1) + m(c-c_1) \geq m(c),
\end{equation}
and so, $E(u) = m(c_1)$, namely $u$ is global minimizer with respect to $c_1$. Thus, by Lemma~\ref{prop1}, v), we have that
$$ m(c) < m(c_1) + m(c-c_1),$$
which contradicts~\eqref{eq15}. Hence, the claim follows and $\|u\|_2^2 =c$. 

\medbreak 
At this point, since $\{v_n\}$ is a bounded sequence in $H^2(\RN)$, it follows from~\eqref{G-N-H2-ineq} and~\eqref{interpolation} respectively that $\|v_n\|_{2\sigma + 2}^{2\sigma + 2} \to 0$ and 
$\|\nabla v_n\|_2^2 \to 0$ as $n \to \infty$.
Thus, we obtain that
\begin{equation}\label{EX1}
 \liminf_{n \to \infty} E(v_n) = \liminf_{n \to \infty} \frac{\gamma}{2} \|\Delta v_n\|_2^2 \geq 0.
\end{equation}
On the other hand, since $\|u\|_2^2= c$, we deduce from~\eqref{eq11} that
$$E(u_n) = E(u) + E(v_n) + o_n(1) \geq m(c) + E(v_n) + o_n(1),$$
and so, that
\begin{equation}\label{EX2}
 \limsup_{n \to \infty} E(v_n) \leq 0.
\end{equation}
From~\eqref{EX1} and~\eqref{EX2} we deduce that $\|\Delta v_n\|_2^2 \to 0$ as $n \to \infty$ and so, that $u_n(\cdot - y_n) \to u$ in $H^2(\R^N).$
\end{proof}

\section{An associated minimization problem} \label{sect-associated-problem}

In this section we present a result we use for the proofs of Theorems~\ref{th1} and~\ref{th2} but of independent interest. 
Moreover, it enlightens the difficulty of the minimization problem for $m(c)$.  Let us introduce 
$$I(u) :=\frac{\gamma}{2}\int_{\R^N}|\Delta u|^2\, dx -\frac{\beta}{2}\int_{\R^N}|\nabla u|^2\, dx,$$
and consider the constrained minimization problem
$$m_I(c) := \inf_{u \in S(c)}I(u).$$ 

\begin{lem}\label{Mihai-1}
For all $c>0$, it follows that:
\begin{itemize}
\item[i)] $m_I(c) = - \frac{\beta^2}{8 \gamma}\, c.$
\item[ii)] The infimum is never achieved.
\item[iii)] All minimizing sequences are vanishing.
\end{itemize}
\end{lem}

\begin{proof}
First observe that if, for some $c>0$, $u$ is a minimizer of $m_I(c)$, then, for any $c_1 >0$, $\left(\frac{c_1}{c}\right)^{1/2}u$ is a minimizer of $m_{I}(c_1)$. Hence, if a minimizer exists for some $c_0 >0$, it exists for any $c>0$. 

\medbreak
Let $c>0$ arbitrary but fixed, for $u \in S(c)$ arbitrary, let us first minimize $I$ along the ray defined by $u_s$. From the definition~\eqref{dilatation-louis} we see that the restriction of $I$ is given by 
\begin{align} \label{functional}
I(u_{s})=\frac{\gamma s^{2}}{2}\int_{\R^N}|\Delta u|^2\, dx -\frac{\beta s}{2}\int_{\R^N}|\nabla u|^2\, dx.
\end{align}
Then, computing the minimum of the function on $s$, one easily gets that 
\begin{equation*} 
 \inf_{ s >0}I(u_{ s}) = -\frac{\beta^2}{8\gamma} \frac{\|\nabla u\|_2^4}{\|\Delta u\|_2^2} 
\end{equation*}
Thus, we deduce that
\begin{equation}\label{expressionduminimumdeI}
m_I(c) =  \inf_{u \in S(c)} \left[-\frac{\beta^2}{8\gamma} \frac{\|\nabla u\|_2^4}{\|\Delta u\|_2^2}  \right],
\end{equation}
or equivalently
\begin{equation*}
\begin{aligned}
m_{I}(c) & = \inf_{u \in S(c)} \left[ -\frac{\beta^2}{8\gamma}\, c\,  \frac{\|\nabla u\|_2^4}{\|\Delta u\|_2^2 \|u\|_2^2} \right]  = \inf_{ u \in H^2(\RN) \setminus \{0\}} \left[ -\frac{\beta^2}{8\gamma}\, c\,  \frac{\|\nabla u\|_2^4}{\|\Delta u\|_2^2 \|u\|_2^2} \right] = -\frac{\beta^2}{8\gamma}\, c \sup_{  u \in H^2(\RN) \setminus \{0\}} R(u),
\end{aligned}
\end{equation*}
where we have set 
\begin{equation*}
R(u) :=  \frac{\|\nabla u\|_2^4}{\|\Delta u\|_2^2 \|u\|_2^2} =  \frac{\left(\int_{\R^N}|\nabla u|^2 \, dx\right)^2}{\left( \int_{\R^N}|\Delta u|^2 \, dx\right) \left( \int_{\RN} |u|^2 dx \right)}.
\end{equation*}
At this point we deduce that $u$ is a minimizer of $m_I(c)$ if and only if it is a maximizer of
\begin{equation*}
 \sup_{u \in H^2(\R^N)\backslash \{0\}} R(u).
\end{equation*}
Let us then show that this supremum is equal to $1$ and never achieved. Using the Fourier transform, we get
\begin{equation*}
R(u) = \frac{\left(\int_{\R^N}|\xi|^2 |\hat{u}(\xi)|^2 \, d \xi\right)^2}{\left(\int_{\R^N}|\xi|^4 |\hat{u}(\xi)|^2 \, dx\right) \left(\int_{\R^N}|\hat{u}(\xi)|^2\, dx\right)}.
\end{equation*}
By Cauchy-Schwarz inequality, 
it follows that $R(u) \leq 1$ (note that this information is precisely~\eqref{infimum}). Let us now construct a sequence $\{u_n\} \subset H^2(\R^N)$ such that $R(u_n) \to 1$. This will prove that 
$$m_I(c) = - \frac{\beta^2}{8 \gamma} \, c.$$
For $\phi \in C_0^{\infty}(\R^N)\backslash \{0\}$ we define the sequence $\phi_n(x) = n^{\frac{N}{2}}\phi(n(x-1))$ and we note that
$\|\phi_n\|_2^2 = \|\phi\|_2^2$, for all $n \in \N$. Now, we define $\{u_n\}$ as  $\hat{u}_n(\xi) = \phi_n(\xi)$.
and we get that
\begin{equation*}
R(u_n) = \frac{\left(\int_{\R^N} |1 + \frac{y}{n}|^2 \phi^2(y)\, dy\right)^2}{\left(\int_{\R^N} |1 + \frac{y}{n}|^4 \phi^2(y)\, dy\right) \|\phi\|_2^2}.
\end{equation*}
By dominated convergence, we obtain $R(u_n) \to 1$ as $n \to \infty$. Hence, $\{u_n\}$ is the desired sequence and i) follows.

\medbreak
Next, let us show that $R(u)=1$ never holds. If we assume by contradiction that there exists a $u \in H^2(\R^N)$ such that $R(u)=1$, this corresponds to the equality case in Cauchy-Schwarz and thus 
\begin{equation*}
|\xi|^4 \,  |\hat{u}(\xi)|^2 = \omega \, |\hat{u}(\xi)|^2 \mbox{ a. e. for some } \omega >0.
\end{equation*}
This is a contradiction with the fact that $\hat{u} \in L^2(\R^N) \backslash \{0\}$.

\medbreak
At this point, i) and ii) have been established. To conclude the proof it remains to show that any minimizing sequence of $m_I(c)$ is vanishing. Indeed, if for some $c>0$, we assume that there exists a non-vanishing minimizing sequence, then, following the proof of Lemma~\ref{alternative}, we get that there exists a $0 < c_1 \leq c$ such that $m_I(c_1)$ is reached. By ii) we know this cannot happen. Hence, the vanishing always holds.
\end{proof}

\begin{remark}\label{bottom-spectre} \label{bellazzini-frank-visciglia} $ $
\begin{itemize}
\item[a)] Let us consider the self-adjoint operator $S = \gamma \Delta^2 + \beta \Delta$ defined in $H^2(\RN)$. As a consequence of Lemma~\ref{Mihai-1}, see also Lemma~\ref{lower-bound}, we deduce that the infimum of the spectrum of $S$ is given by $-\beta^2/4\gamma$ and that it belongs to the essential spectrum. 
\item[b)] The proof of i) and ii) of the previous lemma relies on the fact that there does not exist a maximizer for the interpolation inequality~\eqref{interpolation}. This fact was already observed in~\cite[Example 2.1]{BeFaVi}. 
\end{itemize}
\end{remark}

\section{A Necessary and sufficient condition for the existence of a minimizer} \label{sect-existence}

The aim of this section is to give a necessary and sufficient condition for the existence of a minimizer of $m(c)$. In particular, this condition will show that a minimizer always exists if $c>0$ is sufficiently large. Defining
\begin{equation} \label{c0}
c_0 := \inf\left\{ c>0 : m(c) < - \frac{\beta^2}{8 \gamma}c\right\},
\end{equation}
we have the following result.

\begin{lem}\label{lem:existence} $ $
\begin{itemize}
\item[i)] $c_0 < \infty$.
\item[ii)] If $c_0 =0$, then $m(c) <- \frac{\beta^2}{8 \gamma}c$ and it is reached for any $c>0$.
\item[iii)] If $c_0 >0$, then:
\begin{itemize}
\item[a)] $m(c) = - \frac{\beta^2}{8 \gamma} c$ and it is not reached for any $c \in (0, c_0]$.
\item[b)] $ m(c) < - \frac{\beta^2}{8 \gamma} c$  and it is reached for any $c > c_0$.
\end{itemize}
\end{itemize}
\end{lem}

\begin{proof}
Let us first prove that $c_0 < \infty$. We fix an arbitrary $u \in S(1)$ and we define $u^{\tau} = \tau^{\frac{1}{2}} u(x)$ with $\tau >0$. Then, we have that $u^{\tau} \in S(\tau)$ and
$$ E(u^{\tau}) = \tau \left[ \frac{\gamma}{2}\|\Delta u\|_2^2 - \frac{\beta}{2}\|\nabla u\|_2^2 \right] - \frac{\alpha \tau^{\sigma +1} }{2 \sigma +2} \|u\|_{2 \sigma +2}^{2 \sigma +2}.$$
Since $\sigma +1 >1$, we easily deduce that 
$$ m(\tau) \leq E(u^{\tau}) <  - \frac{\beta^2}{8 \gamma} \tau$$
for $\tau >0$ large enough and so, that $c_0 < \infty$.

\medbreak
Now observe that, if for some $c>0$ the vanishing occurs for a minimizing sequence $\{u_n\} \subset S(c)$, then, by \cite[Lemma I.1]{LiII}, we have that $\|u_n\|_{2\sigma + 2}^{2\sigma+2} \to 0$ as $n \to \infty$ and so, that $m(c) \geq m_{I}(c)$. Hence, the strict inequality $m(c) < m_{I}(c) = -\frac{\beta^2}{8\gamma}c$ guarantees that the vanishing does not happen. Applying then Lemma \ref{alternative}, we deduce that $m(c) < m_{I}(c)$ implies the existence of a global minimizer.

\medbreak
To end the proof just observe that Lemma \ref{prop1}, ii) guarantees that $m(c) < - \frac{\beta^2}{8 \gamma} c$ for any $c >c_0$. On the other hand, if $c \in [0,c_0]$, then $m(c) = m_{I}(c)$ and so, any minimizing sequence of $m(c)$ is a minimizing sequence of $m_{I}(c)$. Thus, by Lemma \ref{Mihai-1}, it must be vanishing and $m(c)$ is not reached.
\end{proof}

\section{Some classes of testing functions} \label{sect-test-functions}

This section is devoted to find sufficient conditions on $\sigma > 0$ ensuring that $c_0 = 0$ in Lemma~\ref{lem:existence}, i.e. guaranteeing that a minimizer exists for any $c > 0$.

\medbreak
We start with an observation which, although not essential in our proofs, simplifies some computations. To that end, let us introduce the constrained minimization problem 
\[ m_{\Phi}(c):= \inf_{u \in S(c)} \Phi(u),\]
where
\[ \Phi(u) = \|(\Delta +1)u \|_2^2 - \frac{\alpha}{2\sigma +2} \|u\|_{2\sigma+2}^{2\sigma+2},\]
and $S(c)$ is given in~\eqref{de:Mmu}.

Let us introduce
\begin{equation} \label{value-tilde-c}
\widetilde{E}(u):= \|\Delta u \|_2^2 - 2 \|\nabla u\|_2^2 - \frac{\alpha}{2\sigma+2} \|u\|_{2\sigma+2}^{2\sigma+2} 
\end{equation} 
and 
\[ 
\tau = \Big( \frac{8 \gamma}{\beta^2}\Big)^{\frac{1}{\sigma}} \Big( \frac{\beta}{2 \gamma}\Big)^{\frac{N}{2}}. 
\]
For a given $c > 0$, the problem of minimizing $E$ on $S(c)$ is equivalent to minimize $\widetilde{E}$ on $S(\tau c)$. Indeed, let $v(x) := bu(ax)$ with 
\[  a = \left( \frac{2\gamma}{\beta} \right)^{\frac{1}{2}} \quad \textup{ and } \quad b = \left( \frac{8\gamma}{\beta^2} \right)^{\frac{1}{2\sigma}},
\]
we obtain that
\[ E(u) = b^{-2-2\sigma} a^N \left[ \|\Delta v\|_2^2 - 2 \|\nabla v\|_2^2 - \frac{\alpha}{2\sigma + 2} \|v\|_{2\sigma+2}^{2\sigma+2} \right] = b^{-2-2\sigma} a^N \widetilde{E}(v),\]
and 
\[ \tau \|u\|_2^2 = b^2 a^{-N} \|u\|_2^2 = \|v\|_2^2.\]
we can assume without loss of generality that $\gamma = 2$ and $\beta = 4$, namely that
\[E(u) = \|\Delta u \|_2^2-2\|\nabla u\|_2^2 - \frac{\alpha}{2\sigma+2} \|u\|_{2\sigma+2}^{2\sigma+2}.\]
The we recognize that $\Phi(u) = E(u) + \|u\|_2^2$. 
\begin{prop} \label{problem-equivalent}
Let $c_0 \in [0, +\infty)$ given in~\eqref{c0}. We have that $c_0 = 0$ if and only if $m_{\Phi}(c) < 0$ for all $c > 0$.
\end{prop}

\begin{proof}
The condition $m(c) < - \frac{\beta^2}{8 \gamma}c$ derived in Lemma~\ref{lem:existence} corresponds to $m_\Phi(c) < 0$. 
\end{proof}

Now, we give sufficient conditions on $\sigma > 0$ ensuring that $m_{\Phi}(c) < 0$ for all $c > 0$. Then, as a consequence of the previous proposition, these conditions guarantee that $c_0 = 0$ in Lemma~\ref{lem:existence}.

\begin{prop} \label{prop mphi}
Assume that $ 0 < \sigma < \max \left\{\frac{4}{N+1},1 \right\}$. Then $m_{\Phi}(c) < 0$ for all $c > 0$.
\end{prop}

Let us split the proof into two lemmas. The proof of both lemmas consists in finding a suitable test function $u \in S(c)$ such that $\Phi(u) < 0$. As mentioned in Remark~\ref{remark test functions}, we will look for functions such that the $L^2-$norm of their Fourier transform concentrates around the unit sphere. In our first construction, we consider a function whose Fourier transform is a perturbed Gaussian centered at $e_1 = (1,0,\ldots,0)$.

\begin{lem}\label{function-test-1}
Assume that $0 < \sigma < \frac{4}{N+1}$. Then $ m_{\Phi}(c) < 0 $ for all $ c > 0 $. 
\end{lem}

\begin{proof} We fix an arbitrary $ c > 0 $ and, for any $ \tau > 0$, we define
\[  u_{\tau}(x) = \pi^{-\frac{N}{4}} c^{\frac 12} \tau^{\frac{N+1}{2}}  e^{ix_1} e^{- \tau^4 x_1^2 + \tau^2 x_2^2 + \cdots \tau^2 x_N^2}.\]
It is clear that $u_{\tau} \in \mathcal{S} ( \R^N )$ and an easy computation gives
$$
\| u_{\tau} \|_{2}^2 =  \pi^{ - \frac N2} c \tau^{N+1} \int_{\R^N} e^{- \tau^4 x_1^2 + \tau^2 x_2^2 + \cdots \tau^2 x_N^2} \, dx 
=  \pi^{ - \frac N2} c \int_{\R^N} e ^{ - |y |^2 } \, dy = c.
$$
Also, we have that
\begin{equation} \label{tf11}
\begin{aligned}
\displaystyle 
\|u_{\tau}\|_{2 \sigma +2}^{2 \sigma +2}
& =  \pi^{ - \frac{N(\sigma + 1)}{2}} c ^{ \sigma + 1} \tau^{(N+1)( \sigma + 1)} \int_{\R^N} e^{- (\sigma+1)(\tau^4 x_1^2 + \tau^2 x_2^2 + \cdots \tau^2 x_N^2)} \, dx 
\\
\displaystyle
& =  \pi^{ - \frac{N(\sigma + 1)}{2}} (\sigma + 1)^{- \frac N2}  c ^{ \sigma + 1} \tau^{(N +1) \sigma} \int_{\R^N} e^{-  | y|^2 }\, dy
=  \pi^{ - \frac{\sigma N }{2}} (\sigma + 1)^{- \frac N2}  c ^{ \sigma + 1} \tau^{(N+1) \sigma}.
\end{aligned}
\end{equation}
It is well-known that the Fourier transform of $f(x) = e^{-\alpha^2|x|^2}$ is given by
$$\wh{f}(\xi) = \Big(\frac{\pi}{\alpha^2}\Big)^{\frac{N}{2}} e^{- \frac{|\xi|^2}{4 \alpha^2}}.$$
Using this fact and the basic properties of the Fourier's transform we get
\[ \wh{u_{\tau}}(\xi) = 2^{\frac{N}{2}}\pi^{\frac{N}{4}} c^{\frac{1}{2}} \tau^{-\frac{N+1}{2}} e^{ - \frac{ \left(\frac{\xi_1-1}{\tau} \right)^2 + \xi_2^2 + \cdots + \xi_N^2   }{2\tau^2}}. \]
Hence, by Plancherel's formula, it follows that 
$$
\begin{aligned}
\displaystyle 
\|\Delta u_{\tau}\|_2^2-2\|\nabla u_{\tau}\|_2^2 + \|u_{\tau}\|_2^2 
& = \frac{1}{(2 \pi)^N} \int _{\R^N} \big( |\xi |^4 - 2 |\xi|^2  +1 \big) | \wh{u_{\tau} }( \xi )|^2 \, d \xi \\
& = \pi^{ - \frac N2} c \tau^{-(N+1)}  \int _{\R^N} ( |\xi |^2 - 1) ^2  e^{ - \frac{ \left(\frac{\xi_1-1}{\tau} \right)^2 + \xi_2^2 + \cdots + \xi_N^2   }{\tau^2}} d \xi.
\end{aligned}
$$
Now, using the changes of variables
\[ \xi_1 = 1 + \tau^2 \eta_1, \qquad \xi_j = \tau \eta_j,\ j = 2, \ldots, N,\]
we obtain that
\begin{equation} \label{tf12}
\begin{aligned}
\|\Delta u_{\tau}\|_2^2-2\|\nabla u_{\tau}\|_2^2 + \|u_{\tau}\|_2^2 
& = \pi^{-\frac{N}{2}} c \int_{\R^N} ( \tau^4 \eta_1^2 + 2\tau^2 \eta_1 + \tau^2 \eta_2^2 + \cdots + \tau^2 \eta_N^2)^2 e^{-|\eta|^2} d\eta \\
& = \pi^{-\frac{N}{2}} c \tau^4 \int_{\R^N} ( \tau^2 \eta_1^2 + 2 \eta_1 +   \eta_2^2 + \cdots +  \eta_N^2)^2 e^{-|\eta|^2} d\eta \\
& \leq A \tau^4, \qquad \forall\ \tau \in (0,1],
\end{aligned}
\end{equation}
for some constant $A > 0$ (independent of $\tau$). Then, by~\eqref{tf11} and~\eqref{tf12}, it follows that
\[
\begin{aligned}
 \Phi(u_{\tau})  & \leq A \tau^4 - \frac{\alpha}{2\sigma + 2} \pi^{-\frac{\sigma N}{2}} (\sigma + 1)^{-\frac{N}{2}} c^{\sigma + 1 }  \tau^{\sigma(N+1)} =: A \tau^4 - B \tau^{\sigma(N+1)},
\end{aligned}
\]
with $B$ positive constant (also independent of $\tau$).
Since $ 0 < \sigma(N+1)  < 4$, we may choose $\tau \in (0, 1]$ sufficiently small so that  $  A \tau^4 - B \tau ^{(N+1) \sigma} < 0 $
and then 
$ m_{\Phi}(c) \leq \Phi( u_{\tau} ) \leq   A \tau^4 - B \tau ^{(N+1) \sigma} < 0$. 
\end{proof}


Now, using a different construction based on the fact that
\begin{equation} \label{psi}
\psi(x) = |x|^{-\frac{N-2}{2}} J_{\frac{N-2}{2}} (|x|),
\end{equation} 
satisfies 
\begin{equation} \label{eq-psi}
(\Delta + 1)\psi = 0, \quad \textup{ in } \RN,
\end{equation}
we enlarge the range on $\sigma > 0$ obtained in Lemma~\ref{function-test-1}. Here $J_{\nu}$ denotes the Bessel function of the first kind and order $\nu$ and we refer for instance to~\cite[Appendix B.4]{Gr} for a proof of~\eqref{eq-psi}. 

\medbreak
Since the asymptotic behavior of $\psi$ will play an important role in our construction, we describe it in the following lemma.

\begin{lem} \label{asymptotic behavior psi}
Assume that $N \geq 2$ and let $\psi$ as defined in~\eqref{psi}. Then:
\begin{itemize}
\item[i)] $\displaystyle \psi(x) \sim \left( \frac{1}{2} \right)^{\frac{N-2}{2}} \frac{1}{\Gamma(\frac{N}{2})}$ as $|x| \to 0$.
\medbreak
\item[ii)] $\displaystyle \psi(x) \sim |x|^{-\frac{N-1}{2}} \cos \left( |x|-\frac{(N-1)\pi}{4} \right)$ as $|x| \to +\infty$.
\end{itemize}
\end{lem}

\begin{proof}
The result immediately follows from the asymptotic behavior of $J_{\nu}(t)$ for $\nu \geq 0$ and $t \geq 0$. We refer for instance to~\cite[Formula $(5.16.1)$ page 134]{Le}.
\end{proof}

\begin{lem}\label{function-test-2}
Assume that $N \geq 4$ and $0 < \sigma < 1$. Then $m_{\Phi}(c) < 0$ for all $c > 0$. 
\end{lem}

\begin{remark}$ $
\begin{itemize}
\item[a)] The construction we do here may also be used for $N \leq 3$. 
Nevertheless, since, for $N \leq 3$, the results we are able to obtain do not improve the ones contained in Lemma~\ref{function-test-1}, we focus on $N \geq 4$. 
\item[b)] Note that, for $N \geq 4$, we cover the full mass subcritical range $0 < \sigma N < 4$. 
\end{itemize}

\end{remark}

\begin{proof}
First of all observe that, for $N \geq 4$, $\frac{1}{N-1} < \frac{4}{N+1}$. Hence, having at hand Lemma~\ref{function-test-1}, we can assume without loss of generality that $\sigma > \frac{1}{N-1}$. Then, for all $m \in \N$, we define 
\begin{equation}
\psi_m(x) = \psi(x) \phi \left( \frac{x}{m} \right)
\end{equation}
where $\psi$ is given in~\eqref{psi} and $\phi \in C^{\infty}(\RN)$ is such that $\phi(x) = 1$ if $|x| \leq 1$, $\phi(x) = 0$ if $|x| \geq 2$ and $0 \leq \phi(x)\leq 1$, for all $x \in \RN$. Note that, in this proof, for any $\delta > 0$, $B_{\delta}$ denotes the ball $B_{\delta}(0)$. We now  split the rest of the proof into several steps. 

\medbreak
\noindent \textbf{Step 1: } \textit{There exist $m_1 \in \N$ and $D_1 > 0$ such that, for all $m \geq m_1$, it follows that $\|\psi_m\|_2^2 \leq D_1 m.$} 
\smallbreak
First observe that, by Lemma~\ref{asymptotic behavior psi}, ii), there exists $R \in \N$ such that, for all $x \in \RN$ with $|x| \geq R$,
\[|\psi(x)| \leq C \, |x|^{-\frac{N-1}{2}},\]
and so, that, for all $m \geq R$,
\[ \|\psi_m\|_2^2 = \int_{\RN} \big| \psi(x) \phi \big( \frac{x}{m} \big)\, \big|^2 dx \leq \int_{B_{2m}} |\psi(x)|^2 dx = \int_{B_R} |\psi(x)|^2 dx + \int_{ B_{2m} \setminus B_R} |\psi(x)|^2 dx  \leq C_1 + C_2 (2m -R).\]
Hence, there exist $m_1 \geq R$ and $D_1 > 0$ such that, for all $m \geq m_1$, 
\begin{equation*} 
 \|\psi_m\|_2^2 \leq D_1 m.
\end{equation*}
\medbreak
\noindent \textbf{Step 2: } \textit{There exists $m_2 \in \N$ and $D_2 > 0$ such that, for all $m \geq m_2$, it follows that $\| (\Delta + 1) \psi_m \|_2^2 \leq D_2 m^{-1}$.}
\smallbreak
First of all, using~\eqref{eq-psi}, one can easily check that 
\[ (\Delta +1) \psi_m(x) = \frac{1}{m^2} \Delta \phi \left( \frac{x}{m} \right) \psi(x) + \frac{2}{m} \nabla \phi\left(\frac{x}{m} \right) \cdot \nabla \psi(x),\]
and so, that
\begin{equation} \label{I0}
\begin{aligned}
\|(\Delta+1)\psi_m\|_2^2 & = \frac{1}{m^4} \int_{\RN} \left| \Delta \phi \left( \frac{x}{m} \right) \psi(x) \right|^2 dx + \frac{4}{m^2} \int_{\RN} \left| \nabla \phi\left(\frac{x}{m} \right) \cdot \nabla \psi(x) \right|^2 dx \\  
& \hspace{4.35cm}+ \frac{4}{m^3} \int_{\RN} \left| \Delta \phi \left( \frac{x}{m} \right) \psi(x) \right| \left| \nabla \phi\left(\frac{x}{m} \right) \cdot \nabla \psi(x) \right| dx \\
& \leq  \frac{1}{m^4} \int_{\RN} \left| \Delta \phi \left( \frac{x}{m} \right) \psi(x) \right|^2 dx + \frac{4}{m^2} \int_{\RN} \left| \nabla \phi\left(\frac{x}{m} \right) \cdot \nabla \psi(x) \right|^2 dx \\
& \hspace{4.35cm} + \frac{4}{m^3} \left( \int_{\RN} \left| \Delta \phi \left( \frac{x}{m} \right) \psi(x) \right|^2 dx \right)^{\frac{1}{2}} \left( \int_{\RN} \left| \nabla \phi\left(\frac{x}{m} \right) \cdot \nabla \psi(x) \right|^2 dx \right)^{\frac{1}{2}} \\
& =: \frac{1}{m^4} I_1 + \frac{4}{m^2} I_2 + \frac{4}{m^3} (I_1)^{\frac{1}{2}} (I_2)^{\frac{1}{2}}.
\end{aligned}
\end{equation}
Then, arguing as in Step 1, we obtain that, for all $m \geq R$,
\begin{equation} \label{I1}
I_1 \leq \|\Delta \phi \|_{\infty}^2 \int_{B_{2m}\setminus B_m} |\psi(x)|^2 dx \leq C_1 m.
\end{equation}
Now, since
\begin{equation*} 
J'_{\nu}(t) = - J_{\nu+1}(t) + \frac{\nu}{t} J_{\nu}(t),
\end{equation*}
for all $\nu \geq 0$ and all $t \geq 0$, see for instance~\cite[$(4)$ page 45]{Wa}, we obtain that
\[ \nabla \psi(x) = 
-\frac{x}{|x|}|x|^{-\frac{N-2}{2}}J_{\frac{N}{2}}(|x|) , \quad \forall\ x \in \R^N \setminus \{0\},
\]
and so, similarly as in Lemma~\ref{asymptotic behavior psi}, there exists $R_2 \in \N$ such that, for all $x \in \RN$ with $|x| \geq R_2,$
\[ |\nabla \psi(x)| \leq C |x|^{-\frac{N-1}{2}}.\]
Hence, we deduce that, for all $m \geq R_2$,
\begin{equation} \label{I2}
I_2 \leq \|\nabla \psi\|_{\infty}^2 \int_{B_{2m} \setminus B_m} |\nabla \psi(x)|^2 dx \leq C  \|\nabla \psi\|_{\infty}^2 \int_{B_{2m} \setminus B_m} |x|^{-(N-1)} dx \leq C_2 m.
\end{equation}
Gathering~\eqref{I0}-\eqref{I2}, we obtain that
\[ \|(\Delta+1)\psi_m\|_2^2 \leq C_1 m^{-3} + 4 C_2 m^{-1} + 4 C_1^{1/2} C_2^{1/2} m^{-2},\]
and so, that there exist $m_2 \geq \max\{R,R_2\}$ and $D_2 > 0$ such that, for all $m \geq m_2$,
\[ \|(\Delta + 1)\psi_m\|_2^2 \leq D_2 m^{-1}.\]
\medbreak
\noindent \textbf{Step 3: } \textit{There exists $D_3 > 0$ such that, for all $m \in \N$, it follows that $\|\psi_m\|_{2\sigma+2}^{2\sigma+2} \geq D_3$.}
\medbreak
Directly observe that, for all $m \in \N$, 
\[ \|\psi_m\|_{2\sigma+2}^{2\sigma+2} = \int_{\RN} \left| \psi(x) \phi\left( \frac{x}{m} \right) \right|^{2\sigma+2} dx \geq \int_{B_1} |\psi(x)|^{2\sigma+2} dx.\]
The claim then follows directly from Lemma~\ref{asymptotic behavior psi}, i).
\medbreak
\noindent \textbf{Step 4: }\textit{Conclusion.}
\medbreak
We fix an arbitrary $c > 0$ and, for any $m \in \N$, we define $\widetilde{\psi}_m = c^{\frac{1}{2}} \frac{\psi_m}{\|\psi_m\|_2}$. It is clear that $\widetilde{\psi}_m \in S(c)$ for all $m \in \N$. Then, by Steps 1, 2 and 3, we deduce that there exists $m_3 \geq \max\{m_1,m_2\}$ such that, for all $m \geq m_3$, 
\begin{equation*}
\begin{aligned}
\Phi(\widetilde{\psi}_m) & = \frac{c}{\|\psi_m\|_2^2} \|(\Delta +1) \psi_m\|_2^2 - \frac{\alpha}{2\sigma+2} \frac{c^{\sigma+1}}{\|\psi_m\|_2^{2\sigma+2}} \|\psi_m\|_{2\sigma+2}^{2\sigma+2}  \\
& \leq \frac{c}{\|\psi_m\|_2^2} D_2 m^{-1} - \frac{\alpha}{2\sigma+2} \frac{c^{\sigma+1}}{\|\psi_m\|_2^{2\sigma+2}} D_3 \\
& = \frac{1}{\|\psi_m\|_2^2} \left[ c D_2 m^{-1} - \frac{\alpha D_3 c^{\sigma+1}}{2\sigma+2} \frac{1}{\|\psi_m\|_2^{2\sigma}} \right] \\
& \leq \frac{1}{\|\psi_m\|_2^2} \left[ c D_2 m^{-1} - \frac{\alpha D_3 c^{\sigma+1}}{(2\sigma+2)D_1^{\sigma}} m^{-\sigma} \right] =: \frac{1}{\|\psi_m\|_2^2} \left[ Am^{-1} - B m^{-\sigma}\right],
\end{aligned}
\end{equation*}
with $A$ and $B$ positive constants independent of $m$. Arguing as in Step 3, one can easily see that $\|\psi_m\|_2^2 \geq D_4 > 0$ for all $m \in \N$. Thus, since $0 < \sigma < 1$, we may choose $m \geq m_3$ sufficiently large so that $A m^{-1} - B m^{-\sigma} < 0$ and then $m_{\Phi}(c) \leq \Phi( \widetilde{\psi}_m) \leq A m^{-1} - B m^{-\sigma } < 0$. 
\end{proof}

\begin{proof}[\textbf{Proof of Proposition~\ref{prop mphi}}]
It follows directly from Lemmas~\ref{function-test-1} and~\ref{function-test-2}.
\end{proof}

At this point we can give the proof of our first main result.

\begin{proof}[\textbf{Proof of Theorem~\ref{th1}}]
The existence part of Theorem~\ref{th1} is a direct consequence of Lemma~\ref{lem:existence} and Propositions~\ref{problem-equivalent} and~\ref{prop mphi}.
Hence, to conclude we just need to obtain the lower bound on the Lagrange multiplier. We recall that if
$u \in S(c)$ is a global minimizer of $m(c)$ (or more generally a constrained critical point), there exists a $\lambda \in \R$ such that $E'(u) = - \lambda u$, namely a $\lambda \in \R$ such that
\begin{equation}\label{E6}
\gamma \Delta^2 u + \beta \Delta u - \alpha |u|^{2 \sigma}u = - \lambda u.
\end{equation}
Multiplying~\eqref{E6} by $u$ and integrating it follows that
\begin{equation}\label{E7}
- \lambda c =   \gamma  \|\Delta u\|_2^2 - \beta \|\nabla u\|_2^2 - \alpha \|u\|_{2 \sigma +2}^{2 \sigma +2}.
\end{equation}
Then, from~\eqref{E7} and Lemma~\ref{Mihai-1}, we deduce that
\begin{align*}
- \lambda c & 
 =    2 E(u) - \frac{2 \alpha \sigma}{2 \sigma +2}\|u\|_{2 \sigma +2}^{2 \sigma +2}  \\
& =  2 m(c) - \frac{2 \alpha \sigma}{2 \sigma +2}\|u\|_{2 \sigma +2}^{2 \sigma +2}  
 < 2 m(c) \leq 2 m_I(c) = - \frac{\beta^2}{4 \gamma}c,
\end{align*}
and thus $\lambda > \displaystyle \frac{\beta^2}{4 \gamma}$.
\end{proof}
\begin{remark}
 At this point of the analysis, we can emphasize that in dimension $N=1$ both sets of test functions provides $\sigma<2$. Indeed, in \textbf{Step 3} of the proof of Lemma~\ref{function-test-2} the lower bound improves to $\|\psi_m\|_{2\sigma+2}^{2\sigma+2} \geq D_3 m$. In the case $N=2$, it improves only to $\|\psi_m\|_{2\sigma+2}^{2\sigma+2} \geq D_3 \log(m+1)$. 

In dimension $3$, both test sets provide the same bound. 
 
Notice that the usual unitary scaling argument provides $\sigma<\frac{2}{N}$ which also coincides with the previous bounds when $N=1$.
\end{remark}

\section{Behavior of the minimizers as $c \to 0$, proof of Theorem~\ref{th2}} \label{sect-limit}

In the section we give the proof of Theorem~\ref{th2}. Recall that $\{(u_n, c_n)\} \subset S(c_n) \times \R$ is such that $c_n \to 0$ and, for each $n \in \N$, $u_n \in 
S(c_n)$ is a minimizer of $m(c_n)$ and $\lambda_n \in \R$ is the associated  Lagrange multiplier. Hence, without loss of generality we may assume that $c_n \leq 1$.

\begin{proof}[\textbf{Proof of Theorem~\ref{th2}}]
Let us start with some preliminary observations. 
\medbreak
\noindent \textbf{Claim 1:} \textit{ $\displaystyle \left\{ \frac{\|\Delta u_n\|_2}{\|u_n\|_2} \right\}$ remains bounded. }
\medbreak
Indeed, using~\eqref{G-N-H2-ineq},~\eqref{interpolation} and Lemma~\ref{prop1}, i), we have that
\begin{equation}\label{E2}
0 \geq 2 E(u_n) \geq  \gamma \|\Delta u_n\|_2^2 - \beta \|\Delta u_n\|_2 \|u_n\|_2 - \frac{\alpha}{\sigma +1}B_N(\sigma)\|\Delta u_n\|^{\frac{\sigma N}{2}} \|u_n\|_2^{2 + 2 \sigma - \frac{\sigma N}{2}},
\end{equation}
and so, 
\begin{equation}\label{E3}
\|\Delta u_n\|_2^2  \leq \frac{\beta}{\gamma} \|\Delta u_n\|_2 \|u_n\|_2 + C \|\Delta u_n\|^{\frac{\sigma N}{2}} \|u_n\|_2^{2 + 2 \sigma - \frac{\sigma N}{2}}.
\end{equation}
Dividing~\eqref{E3} by $\|u_n\|_2^2$, it follows that
\begin{equation}\label{E4}
\left( \frac{\|\Delta u_n\|_2}{\|u_n\|_2}\right)^2 \leq \frac{\beta}{\gamma} \frac{\|\Delta u_n\|_2}{\|u_n\|_2} + C \left( \frac{\|\Delta u_n\|_2}{\|u_n\|_2}\right)^{\frac{\sigma N}{2}} \|u_n\|_2^{2 \sigma} \leq \frac{\beta}{\gamma} \frac{\|\Delta u_n\|_2}{\|u_n\|_2} + C \left( \frac{\|\Delta u_n\|_2}{\|u_n\|_2}\right)^{\frac{\sigma N}{2}}
\end{equation}
Since $0< \sigma N <4$, from~\eqref{E4}, the boundedness of the left hand side follows and thus the claim is proved.
\medbreak
\noindent \textbf{Claim 2:}\textit{ $ \displaystyle \frac{\|u_n\|_{2 \sigma +2}^{2 \sigma +2}}{\|u_n\|_2^2} \to 0$ as $n \to \infty$.
}
\medbreak
By~\eqref{G-N-H2-ineq} we know that
$$ \|u_n\|_{2 \sigma +2}^{2 \sigma +2} \leq B_N(\sigma) \|\Delta u_n\|_2^{\frac{\sigma N}{2}} \|u_n\|_2^{2 + 2 \sigma - \frac{\sigma N}{2}}.$$
Then, dividing by $\|u_n\|_2^2$ one gets
$$ 
\frac{\|u_n\|_{2 \sigma +2}^{2 \sigma +2}}{\|u_n\|_2^2} \leq B_N(\sigma) \left( \frac{\|\Delta u_n\|_2}{\|u_n\|_2}\right)^{\frac{\sigma N}{2}} \|u_n\|_2^{ 2 \sigma} = B_N(\sigma) \left( \frac{\|\Delta u_n\|_2}{\|u_n\|_2}\right)^{\frac{\sigma N}{2}} c_n^{\sigma}
$$
From Claim 1 and the fact that $c_n \to 0$ as $n \to \infty$, the claim follows. Now, we split the rest of the proof into several steps:
\medbreak
\noindent \textbf{Step 1: } \textit{Proof of} i).
\medbreak
It just remains to prove the first inequality. The other statements have already been established, see Section~\ref{sect-existence}. Directly observe that
\[ m(c_n) = E(u_n) = I(u_n) - \frac{\alpha}{2\sigma+2} \|u_n\|_{2\sigma+2}^{2\sigma+2} \geq m_I(c_n) - \frac{\alpha}{2\sigma+2} \|u_n\|_{2\sigma+2}^{2\sigma+2} = m_I(c_n) - \frac{\alpha}{2\sigma+2} \frac{\|u_n\|_{2\sigma+2}^{2\sigma+2}}{\|u_n\|_2^2} c_n.\]
The desired inequality then follows from Claim 2.
\medbreak
\noindent \textbf{Step 2: } \textit{Proof of } ii).
\medbreak
By Theorem~\ref{th1} we know that $\lambda_n > \frac{\beta^2}{4\gamma}$. Now, recalling that
\begin{equation} \label{eqbuena}
- \lambda_n c_n =   \gamma  \|\Delta u_n\|_2^2 - \beta \|\nabla u_n\|_2^2 - \alpha \|u_n\|_{2 \sigma +2}^{2 \sigma +2},
\end{equation}
and using Lemma~\ref{Mihai-1} and Claim 2, one can write
$$ - \lambda_n c_n \geq 2 m_I(c_n) - \alpha \|u_n\|_{2 \sigma +2}^{2 \sigma +2} = - \frac{\beta^2}{4 \gamma} c_n - \frac{\|u_n\|_{2\sigma+2}^{2\sigma+2}}{\|u_n\|_2^2} c_n =  - \frac{\beta^2}{4 \gamma} c_n- \mu_n c_n$$
for a sequence $\mu_n \to 0$ as $n \to \infty.$ Thus, we have that 
$$\frac{\beta^2}{4\gamma} < \lambda_n \leq \frac{\beta^2}{4 \gamma} + \mu_n$$
for a sequence $\mu_n \to 0$ and ii) follows.
\medbreak
\noindent \textbf{Step 3: } \textit{Proof of }iii) 
\medbreak
By~\cite[Lemma 2.1]{BoCaGoJe-2}, we have that $u_n \in S(c_n)$ satisfies the Pohozaev type identity
\begin{equation}\label{E8}
\gamma \|\Delta u_n\|_2^2 - \frac{\beta}{2}\|\nabla u_n\|_2^2 - \frac{\sigma N}{2 (2 \sigma +2)}\|u_n\|_{2 \sigma +2}^{2 \sigma +2}=0.
\end{equation}
Dividing~\eqref{E8} by $\|u_n\|_2^2$ and using Claim 2 we get that
\begin{equation}\label{E9}
\gamma \frac{\|\Delta u_n\|_2^2 }{\|u_n\|_2^2 } - \frac{\beta}{2} \frac{\|\nabla u_n\|_2^2 }{\|u_n\|_2^2 } \to 0.
\end{equation}
Also, by i), we know there exists $\mu_n \to 0$ as $n \to \infty$ such that
\begin{equation*}
E(u_n) = \frac{\gamma}{2} \|\Delta u_n\|_2^2 - \frac{\beta}{2} \|\nabla u_n\|_2^2 - \frac{\alpha}{2 \sigma +2} \|u_n\|_{2 \sigma +2}^{2 \sigma +2} =  - \frac{\beta^2}{8 \gamma}\|u_n\|_2^2 + \mu_n \|u_n\|_2^2 - \frac{\alpha}{2 \sigma +2} \|u_n\|_{2 \sigma +2}^{2 \sigma +2}.
\end{equation*}
Then, using again Claim 2, we deduce that
\begin{equation}\label{E11}
\frac{\gamma}{2} \frac{\|\Delta u_n\|_2^2 }{\|u_n\|_2^2 }  - \frac{\beta}{2} \frac{\|\nabla u_n\|_2^2 }{\|u_n\|_2^2 } \to - \frac{\beta^2}{8 \gamma}.
\end{equation}
Having at hand~\eqref{E9} and~\eqref{E11} one easily deduce iii).
\medbreak
\noindent \textbf{Step 4: } \textit{Proof of }iv).
\medbreak
Let us define $v_n = \frac{u_n}{\|u_n\|_2}$ for all $n \in \N$. Since $u_n \in S(c_n)$ satisfies~\eqref{eqbuena} and by ii) we know that $\lambda_n \to \frac{\beta^2}{4 \gamma}$, as $n \to \infty$, we get
\begin{equation}\label{E12}
\gamma \|\Delta v_n\|_2^2 - \beta \|\nabla v_n\|_2^2 + \frac{\beta^2}{4 \gamma} \|v_n\|_2^2 = \alpha \frac{\|u_n\|_{2 \sigma +2}^{2 \sigma +2}}{\|u_n\|_2^2} + \mu_n \|v_n\|_2^2
\end{equation}
for some $\mu_n \to 0$. On the other hand, by Plancherel's formula, it follows that
\begin{equation}\label{E13}
\gamma \|\Delta v_n\|_2^2 - \beta \|\nabla v_n\|_2^2 + \frac{\beta^2}{4 \gamma} \|v_n\|_2^2 = \frac{1}{(2 \pi)^N}
\int _{\R^N} \Big( \gamma |\xi |^4 - \beta |\xi|^2 +  \frac{\beta^2}{4 \gamma} \Big)\, | \wh{v_{n} } ( \xi )|^2 \, d \xi .
\end{equation}
Gathering~\eqref{E12} and~\eqref{E13} we have that,
\begin{equation}\label{E14}
\frac{1}{(2 \pi)^N} \int _{\R^N} \Big( \sqrt{\gamma} |\xi |^2 - \frac{\beta}{2 \sqrt{\gamma}} \Big)^2 | \wh{v_{n} } ( \xi )|^2 \, d \xi = \alpha \frac{\|u_n\|_{2 \sigma +2}^{2 \sigma +2}}{\|u_n\|_2^2} + \mu_n  
\end{equation}
where $\mu_n \to 0$ as $n \to \infty$. Using Claim 2 we see that the right hand side goes to $0$ as $n \to \infty$. This proves that~\eqref{4} holds. Now, from~\eqref{4} and~\eqref{E13}, we deduce that
$$ I(v_n) = \frac{\gamma}{2} \|\Delta v_n\|_2^2 - \frac{\beta}{2} \|\nabla v_n\|_2^2 \to - \frac{\beta^2}{8 \gamma}.$$
In view of Lemma~\ref{Mihai-1}, 1), we have that $\{v_n\} \subset S(1)$ is a minimizing sequence for $m_I(1)$. Then, by Lemma~\ref{Mihai-1}, iii), it follows that $\{v_n\}$ is a vanishing sequence. Applying then~\cite[Lemma I.1]{LiII}, we deduce that $v_n \to 0$ in $L^p(\RN)$ for all $p \in (2,4^{\ast})$. This completes the proof of the theorem.
\end{proof}

\section{Some extensions and related problems.} \label{sect-extensions}

In this last section we make some additional remarks and discuss possible extensions of our results.

\subsection{Symmetry of the minimizers in Theorem~\ref{th1}} $ $
\medbreak
As a consequence of the arguments developed in the preprint~\cite{BuLeSo}, when $\sigma  \in \N$, we can obtain the following description of the minima obtained in Theorem~\ref{th1}.  

\begin{prop}\label{Lenzmann-input}
Let $\sigma \in \N$ and, for any arbitrary $c>0$ such that $m(c)$ is reached, let $Q$ be a  minimizer of $m(c)$. Then, it follows that
$$Q(x) = e^{i \tau} Q^{\bullet}(x+x_0),$$
with some constants $\tau \in \R$ and $x_0 \in \R^N$. Here, $Q^{\bullet} : \R^N \to \Cc$ is a smooth bounded and positive definitive function in the sense of Bochner. As a consequence, it holds that
$$ Q^{\bullet}(-x) = \overline{Q^{\bullet}(x)} \quad \mbox{and} \quad Q^{\bullet}(0) \geq |Q^{\bullet}(x)| \quad \mbox{for all} \quad x \in \R^N.$$
\end{prop}

A function $f:\R^N \to\Cc$ is positive definite in the sense of Bochner if for any
$(x_1,\ldots,x_n)$ in $(\R^N)^n$ the matrix 
$(f(u_i-u_j))_{1\leq i,j\leq n}$ 
is definite positive.

\begin{remark}
Our proof is a essentially a consequence of~\cite[Theorem 2]{BuLeSo}. We provide some details for the reader's sake in trying to keep the notation introduced by the authors in~\cite{BuLeSo}.
\end{remark}

\begin{proof}
Let $Q$ be a minima for $m(c)$ and let $\lambda \in \R$ be the associated Lagrange multiplier. By Theorem~\ref{th1}, we know that $\lambda > \beta^2/4\gamma$. Hence, by~\cite[Theorem 3.10]{BoCaMoNa}, we have that $e^{a \, |\, \cdot\,|} Q \in L^2(\RN)$ for some $a > 0$. Note that such decay may be also obtained as in~\cite[Theorem 3]{BuLeSo}. Now, defining
\[ Q^{\bullet}:= \cF^{-1} (|\cF Q|),\] 
we observe, from~\cite[Lemma 2.1]{BuLeSo}, that
\[ \|\Delta Q^{\bullet}\|_2 = \|\Delta Q\|_2, \quad \|\nabla Q^{\bullet}\|_2 = \|\nabla Q\|_2, \quad \|Q^{\bullet}\|_2 = \|Q\|_2 \quad \textup{ and } \quad \|Q^{\bullet}\|_{2\sigma+2} \geq \|Q\|_{2\sigma+2}.\]
Thus, we easily deduce that $Q^{\bullet}$ is also a minima for $m(c)$. Having at hand the suitable exponential decay of $Q$ and the fact that $Q^{\bullet}$ is also a minima for $m(c)$, the rest of the proof follows repeating almost verbatim the proofs of~\cite[Lemmata 2.2 \& 4.1]{BuLeSo}.
\end{proof}

\begin{remark}
The conclusions of Proposition~\ref{Lenzmann-input} hold for any $c>0$,  if $N=1$ or if $N=2$ and $\sigma =1$. In particular, we cover the physical relevant case $N=2$ for the Kerr nonlinearity. 

As mentionned in~\cite[Remark after Lemma 2.1]{BuLeSo}, the method does not extend immediately to non integer~$\sigma$.
\end{remark}

\subsection{Optimal range of $\pmb{\sigma >0}$}$ $ 
\medbreak
The condition $0 < \sigma < \max\left\{4/(N+1),1 \right\}$ is the consequence of two particular trials of test functions. It would be nice to put in light an optimal upper bound on $\sigma >0$,  which permits a minimizer to exists for every $c > 0$. See Figure~\ref{figure}.

{\centering
\begin{figure}[h!] 
\begin{tikzpicture}[line width = 0.45mm, scale = 1.25, >=stealth]
      \draw[->] (1,0) -- (6.55,0) node[right] {$N$};
      \draw[->] (1,0) -- (1,4.4) node[above] {$\sigma$};
      \draw[scale=1,domain=1:6.5,smooth,variable=\x,black,line width = 0.8mm] plot ({\x},{4/\x});
      \draw [black](6.5, 4/6.5) node[right]{4/N};
      \draw[scale=1,domain=1:3,smooth,variable=\x,green!45!black,line width = 0.8mm] plot ({\x},{4/(\x+1)});
       \draw[-, green!45!black,line width = 0.8mm] (3,1) -- (4,1); 
      \draw (4.12,1.12)[green!45!black] node[right]{$\max\left\{\frac{4}{N+1},1\right\}$};      
      \draw (1,-0.1) node[below] {$N=1$};
      \draw[dashed, purple, line width = 0.35mm] (2,0)--(2,2);
      \draw (2,-0.1) node[below] {$N=2$};
      \draw[dashed, purple, line width = 0.35mm] (3,0)--(3,1);
      \draw (3,-0.1) node[below] {$N=3$};
      \draw[dashed, purple, line width = 0.35mm] (4,0)--(4,1);
      \draw (4,-0.1) node[below] {$N=4$};
      \draw[dashed, purple, line width = 0.35mm] (5,0)--(5,4/5);
      \draw (5,-0.1) node[below] {$N=5$};
      \draw[dashed, purple, line width = 0.35mm] (6,0)--(6,2/3);
      \draw (6,-0.1) node[below] {$N=6$};
      

\fill[green, opacity = 0.2] (1,0) -- (6.5,0) -- (6.5,4/6.5) -- (5,4/5) -- (4,1) -- (3,1) -- (2.5, 4/3.5)-- (2,4/3) -- (1.5, 4/2.5) -- (1,2) -- cycle;
\fill[cyan, dashed, opacity = 0.2] (1,2) -- (1.5,4/2.5) -- (2, 4/3) -- (2.5, 4/3.5) -- (3,1) -- (4,1) -- (3.75, 4/3.75) -- (3.5,4/3.5) -- (3.25, 4/3.25) -- (3,4/3) -- (2.75,4/2.75) -- (2.5, 4/2.5) -- (2.25, 4/2.25) -- (2,2) -- (1.75, 4/1.75) -- (1.5,4/1.5) -- (1.25, 4/1.25) -- (1,4) -- cycle;

     
      \draw [-] (0.95,1)--(1.05,1);
      \draw [-] (0.95,2)--(1.05,2);
      \draw [-] (0.95,3)--(1.05,3);
      \draw [-] (0.95,4)--(1.05,4);
      
      \draw (0.94,1) node[left] {1};  
      \draw (0.94,2) node[left] {2};
      \draw (0.94,3) node[left] {3};
      \draw (0.94,4) node[left] {4};
\end{tikzpicture} 
\caption{The existence of a global minimizer on $S(c)$ for $c > 0$ small is open in the smaller region.} 
\label{figure}
\end{figure}
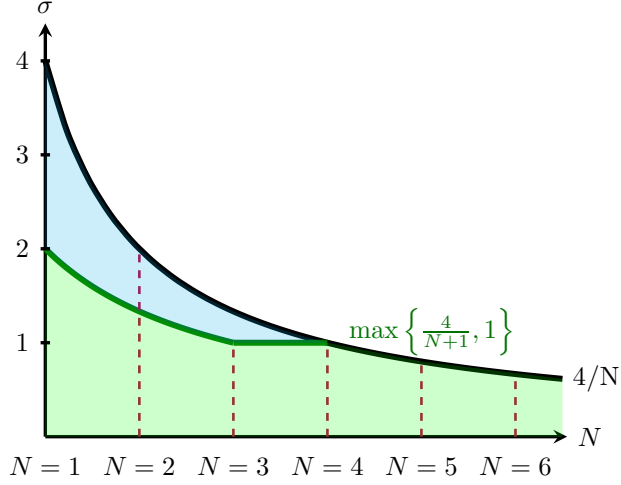
}
More generally, the question to consider seems to be the following : Let $ 0< \sigma N < 4^*$ and define
$$A(1) := \{u \in S(c) : \|\Delta u\|_2^2 \leq 1\}.$$
Which condition on $\sigma >0$ guarantees that, for any $c>0$,
\begin{equation*}
\inf_{u \in S(c) \cap A(1) }I(u) <  \inf_{u \in S(c) \cap A(1) }E(u) \, \, ?
\end{equation*}

\subsection{The mass critical case} $ $
\medbreak

As already mentioned, the results we have derived in the mass subcritical case are also useful in the mass critical case. First, we note the following

\begin{lem} \label{thm1}
Assume $\gamma >0,$ $\beta >0,$ $\alpha >0,$ $\sigma N =4$ and $N \geq 1$. There exists $c^{*}_N >0$ such that
$m(c) \in (- \infty, 0)$ if $c \in (0, c^*_N)$ and $m(c) = - \infty$ if $c \geq c^*_N$. Actually  
\begin{align} \label{cn}
c^{*}_N=  \left(\frac{\gamma}{\alpha} C(N) \right)^{\frac N4} \quad \mbox{where} \quad
C(N):=\frac{N+4}{N B_N (\frac 4N) },
\end{align}
and $B_N (\sigma)$ is the smallest constant satisfying ~\eqref{G-N-H2-ineq}. In addition, $E$ is coercive on $S(c)$ if $c \in (0, c^*_N)$. 
\end{lem}
\begin{proof}
First note that, when $\sigma N =4$, one has
\begin{equation}\label{ajout2}
\frac{ \alpha N}{N+4}\|u\|_{2 +\frac{8}{N}}^{2 +\frac{8}{N}} \leq  \left(\frac {c}{c^*_N}\right)^{\frac 4N}  \gamma \|\Delta u\|_2^2,  \quad \forall\ u \in S(c).  
\end{equation}
Indeed,~\eqref{ajout2} follows from the Gagliardo-Nirenberg inequality~\eqref{G-N-H2-ineq} using the definition  of $c_N^*$ given in~\eqref{cn}. Now, using~\eqref{interpolation} and~\eqref{ajout2}, we have that
\begin{align*} 
\begin{split}
E(u)& \geq \frac {\gamma}{2} \|\Delta u\|_2^2 - \frac{\beta}{2} \|\Delta u\|_2 \|u\|_2 
     - \frac{\alpha N}{2N+8} \|u\|_{2 + \frac 8N}^{2 + \frac 8N}  \\
    & \geq \frac {\gamma}{2} \left(1-\left(\frac{c}{c_N^*}\right)^{\frac 4N}\right) \|\Delta u\|_2^2 - \frac{\beta }{2} c^{\frac{1}{2}} \|\Delta u\|_2,  \quad \qquad \forall\ u \in S(c).
\end{split}
\end{align*}
Therefore, we deduce that $E$ is coercive if $c < c^*_N$ and then, in particular, that $m(c) > - \infty$. The fact that $m(c) <0$ when $c \in (0,c_N^*)$ follows directly from 
\eqref{functional-louis} letting, for an arbitrary $u \in S(c)$, $s \to 0$. 
\medbreak
Now, let us prove that $m(c) = - \infty$ for $c \geq c^*_N$. It follows 
from~\cite{BoLe}, see also~\cite{BeFaVi}, that the best constant $B_N(\frac 4N)$ in~\eqref{G-N-H2-ineq} is achieved, i.e. there exists $U \in H^2(\R^N)$ satisfying
\begin{align} \label{u}
\|U\|_{2 + \frac 8N}^{2 + \frac 8N}=B_N\big(\frac 4N\big)\,\|U\|_2^{\frac 8N}\|\Delta U\|_2^{2}.
\end{align}
Choosing
\begin{align*} 
w:=c^{\frac 12} \frac{U}{\|U\|_2} \in S(c),
\end{align*}
and taking~\eqref{functional-louis} and~\eqref{u} into account, we get
\begin{equation} \label{key10}
\begin{aligned}
E(w_s) &=\frac {c}{2\|U\|_2^2} s^2 \gamma \|\Delta U\|_2^2
     - \frac {c}{2\|U\|_2^2} s \beta  \|\nabla U\|_2^2 
     -\frac{N}{2N+8} \left(\frac {c^{\frac 12}}{\|U\|_2}\right)^{2 + \frac 8N} s^2 \alpha \|U\|_{ 2 + \frac 8N}^{ 2 + \frac 8N} \\
     &= \frac {c}{2\|U\|_2^2} \gamma \left(1-\left(\frac{c}{c_N^*}\right)^{\frac 4N}\right) s^2
     \|\Delta U\|_{2}^2
     - \frac {c}{2\|U\|_2^2} s \beta \|\nabla U\|_2^2  \\
		 & \leq - \frac {c}{2\|U\|_2^2} s \beta \|\nabla U\|_2^2 
\end{aligned}
\end{equation}
which implies that $ E(w_s) \rightarrow -\infty$ as $s \rightarrow \infty$ for any $c \geq c^*_N$.  
\end{proof}

In view of Lemma~\ref{thm1} we obtain the following result on the line of Theorem~\ref{th1}.
\begin{thm}\label{th11}
Assume $\gamma >0$, $\beta >0$, $\alpha >0$, $\sigma N =4$ and $N \geq 5.$ Then for any $c \in (0, c_N^*)$,  any minimizing sequence of $m(c)$ is precompact in $H^2(\R^N)$ up to translations. In particular, $m(c)$ is achieved. In addition, if $u \in S(c)$ is a minimizer of $m(c)$, the associated Lagrange multiplier $\lambda \in \R$  satisfy $\lambda > \frac{\beta^2}{4 \gamma}.$
\end{thm}
\begin{proof}
We know, from Lemma~\ref{thm1}, that $E$ is coercive on $S(c)$ for any  $c \in (0, c_N^*)$. Thus, the arguments developed in the proof of Theorem~\ref{th1} remain valid. Note also that, when $N \geq 5$, we have $\sigma <1$ and it guarantees that the minimizing sequences do not vanish.
\end{proof}

\begin{remark}
It should be possible, when $N \leq 4$, to derive a lower bound $ \tilde{c}_N > 0$ such that, for any $c \in (\tilde{c}_N, c_N^*)$,
 the conclusions of Theorem~\ref{th11} holds. We refer to~\cite{LuZhZh} for elements in that direction.
\end{remark}

\subsection{Bifurcation from the infimum of the essential spectrum} $ $
\medbreak
First observe that, by Theorem~\ref{th2}, iii), we know that, when the minimizers for $m(c)$ exist, they converge to $0$ in the $H^2(\RN)$ norm as $c \to 0$. Also, by Theorem~\ref{th2}, ii), we know that the associated Lagrange multipliers converge to $\beta^2/4\gamma$. Hence, in view of Lemma~\ref{Mihai-1} and  Remark~\ref{bottom-spectre},  we may consider a bifurcation phenomenon from the bottom of the essential spectrum of the operator $u \mapsto \gamma \Delta^2 u + \beta \Delta u$. 
\medbreak
In~\cite[Theorem 1.2]{BoCaMoNa}, see also~\cite{BoNa}, the authors show that when $0 < \sigma N < 4^*$ the equation
\begin{equation}\label{eq:free}
\gamma \Delta^2 u + \beta \Delta u + \lambda u = |u|^{2 \sigma} u, 
\end{equation}
admits a ground state solution  $u_{\lambda} \in H^2(\R^N)$ whenever $\beta < 2 \sqrt{\gamma \lambda}$, namely if $\lambda > \frac{\beta^2}{4 \gamma}$. By a ground state it is intend here a least energy solution for the free functional
$$\E(u)  = \frac{\gamma}{2} \|\Delta u\|_2^2 - \frac{\beta}{2}\|\nabla u\|_2^2 + \frac{\lambda}{2}
\|u\|_{2}^{2}  - \frac{1}{2 \sigma +2} \|u\|_{2 \sigma +2}^{2 \sigma +2}.$$
Note also that, when $\lambda < \frac{\beta^2}{4 \gamma}$, it is expected that~\eqref{eq:free} has no solutions in $H^2(\R^N)$, see~\cite{BoCaMa} in that direction. Worth of interest, in our opinion, would  be to investigate under which conditions on $\sigma \in (0, 4^*/N)$,  the ground states solutions $u_{\lambda}$ to~\eqref{eq:free} satisfy $u_{\lambda} \to 0$ in $H^2(\R^N)$ as $\lambda \to \beta^2/4\gamma$ from the right, thus presenting a bifurcation phenomenon.
\medbreak
This kind of questions were first addressed in the 80's by C. A. Stuart~\cite{St-88} for equations whose model is given by
\begin{equation}\label{stuart}
 - \Delta u - \lambda u = |u|^{2\sigma} u, \quad u \in H^1(\R^N)\,.
\end{equation}
Here the bottom of the essential spectrum is $\lambda =0$. 
For~\eqref{stuart}, the existence of a sequence of solutions $(u_n, \lambda_n) \subset H^1(\R^N) \times (0, + \infty)$ with $u_n \to 0$ in $H^1(\R^N)$ and $\lambda_n \to 0$ was established under the condition that $0 < \sigma N <2$.
 Note that this condition is somehow optimal since, for~\eqref{stuart}, it corresponds to the range for which the associated functional is coercive and no local minima structure is present if $\sigma N >2$. C. A. Stuart analysis relies on the control of the ground state level by the use of suitable test functions. We conjecture that such bifurcation phenomenon will take place for~\eqref{eq:free} not only when $\sigma N <4$, under 
the conditions ensuring the existence of a global minimizer for $m(c)$ when $c>0$ is small, but also when $\sigma N >4$ at least when $\sigma \in (\frac{4}{N}, 1)$. In this second case our conjecture is supported by what was observed in~\cite{BeBoJeVi}.

\bigskip

\noindent \textbf{\emph{Acknowledgements}} : The authors thank very warmly Mihai Mari\c{s}. Without several of his key remarks the work would not have been the same. Part of this work was done during a visit of the second author to the Universit\'e Paul Sabatier (Toulouse). He thanks his hosts for the warm hospitality and the financial support.

\bibliographystyle{plain}
\bibliography{Bibliography}
\vspace{0.25cm}
\end{document}